\documentclass[german, 12pt]{article}

\usepackage{graphicx} 
\usepackage{url}      
\newcommand{\doi}[1]{\url{http://dx.doi.org/#1}}
\usepackage{amsmath}  
\usepackage{amssymb}
\usepackage{mathrsfs}

\usepackage{xspace,algorithm,algorithmic,tabularx}

\usepackage{setspace}
\usepackage{color}
\usepackage{pdflscape}
\usepackage{graphics}
\usepackage{rotating}
\usepackage{verbatim}
\usepackage{amsmath}
\usepackage{amssymb}
\usepackage{amsfonts}
\usepackage{mathrsfs}
\usepackage{array}
\usepackage{epic,eepic}
\usepackage{fancyhdr}
\usepackage{url}
\usepackage{subfigure}

\newtheorem{theorem}{Theorem}[section]
\newtheorem{definition}{Definition}[section]
\newtheorem{remark}{Remark}[section]
\newtheorem{lemma}{Lemma}[section]
\newtheorem{corollary}{Corollary}[section]

\newtheorem{assumptions}{Assumption}[section]

\newenvironment{proof}{\begin{trivlist}
\item[\hskip\labelsep{\it Proof.}]}{$\hfill\Box$ \end{trivlist}}

\DeclareMathOperator{\decay}{decay}

\DeclareMathOperator{\cost}{cost}

\DeclareMathOperator{\CD}{CD}
\DeclareMathOperator{\bias}{bias}

\def\bsa{\boldsymbol{a}}

\def\bsx{\boldsymbol{x}}
\def\bsy{\boldsymbol{y}}

\def\bsk{\boldsymbol{k}}
\def\bsl{\boldsymbol{l}}

\def\bs0{\boldsymbol{0}}
\def\bsq{\boldsymbol{q}}

\def\bsell{\boldsymbol{\ell}}

\def\bsgamma{\boldsymbol{\gamma}}

\def\bstau{\boldsymbol{\tau}}

\def\bsqs12{\boldsymbol{q}_{[s_2] \setminus [s_1]}}
\def\Gs12{G^{s_2-s_1}_{b,m}}

\newcommand{\X}{{\mathfrak X}}
\newcommand{\U}{{\mathcal U}}
\newcommand{\N}{\mathbb{N}}

\newcommand{\R}{\mathbb{R}}
\newcommand{\EE}{\mathbb{E}}
\newcommand{\PP}{\mathbb{P}}
\newcommand{\FF}{\mathbb{F}}
\def\nn{\mathbb{N}}

\newcommand{\nat}{\mathbb{N}}

\newcommand\Var{\textnormal{Var}}

\newcommand{\rr}{\mathbb{R}}

\newcommand{\wal}{{\rm wal}}

\newcommand{\rd}{\, \mathrm{d}}

\newcommand{\bszero}{\boldsymbol{0}}

\allowdisplaybreaks

\title{Optimal randomized changing dimension algorithms for infinite-dimensional integration on function spaces with ANOVA-type decomposition}

\author{Josef Dick\thanks{School of Mathematics and Statistics,
University of New South Wales, Sydney ({\tt josef.dick@unsw.edu.au}).}
\and Michael Gnewuch\thanks{School of Mathematics and Statistics,
University of New South Wales, Sydney,
({\tt m.gnewuch@unsw.edu.au}).}}

\date{}

\begin{document}
\maketitle

\vspace{-15pt}

\emph{\begin{center}Dedicated to Erich Novak on the occasion of his  60th
birthday\end{center}}

\vspace{5pt}

\begin{abstract}
We study the numerical integration problem for functions with infinitely many variables. The function spaces of integrands we consider are weighted reproducing kernel Hilbert spaces with norms related to the ANOVA decomposition of the integrands. The weights model the relative importance of different groups of variables.
We investigate randomized quadrature algorithms and measure their quality by estimating the randomized worst-case integration error.

In particular, we provide lower error bounds for a very general class of randomized algorithms that includes
non-linear and adaptive algorithms. Furthermore, we propose new randomized changing dimension algorithms and present favorable upper error bounds.
For product weights and finite-intersection weights our lower and upper error bounds match and show that our changing dimension algorithms
are optimal in the sense that they achieve convergence rates arbitrarily close to the best possible convergence rate.
As more specific examples, we discuss unanchored Sobolev spaces of different degrees of smoothness and randomized changing dimension algorithms
that use as building blocks scrambled polynomial lattice rules.

Our analysis extends the analysis given in [J. Baldeaux, M. Gnewuch. Optimal randomized multilevel algorithms for infinite-dimensional
integration on function spaces with ANOVA-type decomposition. arXiv:1209.0882v1 [math.NA], Preprint 2012].
In contrast to the previous article we now investigate a different cost model for algorithms.
With respect to that cost model, randomized multilevel algorithms cannot, in general, achieve optimal convergence rates, but, as we prove for
important classes of weights, changing dimension algorithms actually can.

\end{abstract}

 \vspace*{2cm}

\section{Introduction} \label{introduction}
Integrals over functions with an a priori unlimited or even infinite
number of variables appear in applications such as molecular
chemistry, physics or quantitative finance, see, e.g.,
\cite{Gil08a, WW96} and the literature mentioned therein.

Recently a large amount of research has been done on how to solve such
kind of integrals efficiently with the help of cleverly designed algorithms, such as
\emph{multilevel algorithms}, \emph{changing dimension algorithms}, and
\emph{dimension-wise quadrature methods}.
Multilevel Monte Carlo algorithms were introduced by Heinrich \cite{Hei98}
for the computation of solutions of integral equations and by Giles
\cite{Gil08a} for path simulation of stochastic differential equations.
Changing dimension algorithms for infinite-dimensional integration were introduced
by Kuo et al. \cite{KSWW10a}, and dimension-wise quadrature methods for
multivariate integration were introduced by Griebel and Holtz in \cite{GH10}.
Changing dimension algorithms and dimension-wise quadrature methods try to address
the important components of the anchored decomposition of the integrand.

There is a large number of complexity theoretical articles that study the tractability
of infinite-dimensional problems, see, e.g. \cite{WW11a, WW11b, Was12} for function approximation and
\cite{HW01, NH09, HMNR10, KSWW10a, NHMR11, PW11,
Bal12, Gne10, BG12, DG12, Gne12a} for integration.
These results rely strongly on function space
        decompositions of weighted reproducing kernel Hilbert spaces
        like anchored or ANOVA decompositions
        (see, e.g., \cite{KSWW10b})
        and on randomized and deterministic low-discrepancy point sets,
        lattice rules
        or sparse grid constructions.

The ANOVA decomposition is of particular interest for the following reason: Looking at the classical theory of quasi-Monte Carlo (QMC) integration, many researchers expected that QMC integration based on low-discrepancy sequences is not very helpful in higher dimension. But then numerical experiments clearly showed that QMC methods are superior to Monte Carlo (MC) methods for many financial applications in dimensions
as high as $360$ or even higher, see, e.g., \cite{PT95, NT96, PT96, CMO97, ABG98}. One approach to explain these unexpected results is that although the underlying problems are high-dimensional, their \emph{effective dimension} is indeed very small, see \cite{CMO97}. The definition of effective (truncation or
superposition) dimension is based on the ANOVA decomposition of the
integrand. Essentially, a multivariate integrand has low effective
dimension if its variance is sharply concentrated in its lower-order ANOVA terms.
Nevertheless, this argument alone is not completely satisfactory,
since the integrands appearing in finance applications have
usually not the necessary smoothness required for the theoretical
results of the theory of QMC methods.
But Liu and Owen \cite{LO06},
and Griebel, Kuo, and Sloan \cite{GKS10, GKS11}
showed that the ANOVA decomposition has a favorable smoothing effect:
the lower-order ANOVA terms of a function
exhibit more smoothness than the function itself.

Thus it would be highly desirable, especially for finance applications,
to have some  algorithm that addresses the (more
important) lower order ANOVA terms of an integrand by quadratures
that exploit the smoothness of these terms and lead to higher order
convergence, and the (less important) higher order ANOVA terms
by quadratures that take (efficiently) care of the less smoothness terms.
Calculating the ANOVA decomposition of a given integrand is too
expensive, since this requires in particular the exact calculation
of the integral one wants to approximate.
Thus the canonical choice to address the ANOVA terms is to (essentially) utilize
an anchored decomposition.
This was done in \cite{GH10} as well as in \cite{HMNR10, BG12}.

The article \cite{GH10} provides adaptive and non-adaptive algorithms for
high-dimen\-sional integration which perform well in finance applications.
The non-adaptive algorithms are similar to the
changing dimension algorithms from \cite{KSWW10a, PW11}.
Interesting concepts proposed in \cite{GH10} are, e.g., the truncation and superposition
dimension in the anchored case.
Unfortunately, the error analysis in \cite{GH10} does not show how well
the non-adaptive algorithms address the important ANOVA components of integrands, since the provided
error estimates are for
norms based on anchored decompositions and not on ANOVA decompositions.
Furthermore, the assumptions in \cite{GH10} on the computational costs are rather optimistic,
since it is assumed that the cost for function evaluations does not depend
on the number of variables. Here it is more realistic to assume that the cost
depends (at least) linearly on the number of variables, as it is done in the more specific
context of path simulation, see, e.g., \cite{Gil08a}.

In \cite{HMNR10, BG12} the convergence rates of randomized multilevel algorithms for
infinite-di\-men\-sio\-nal integration are analyzed for norms based on the ANOVA decomposition.
The cost model considered in these two articles take into account that function evaluations
are more expensive if more variables are involved  (i.e., more variables are ``active'').
It was shown in \cite{BG12} that
suitable randomized multilevel algorithms achieve the optimal rate of convergence in
this cost model.

In this article we extend the analysis from \cite{HMNR10, BG12}. For our lower bounds we study rather
general randomized integration algorithms, for our upper error bounds we focus on
\emph{randomized changing dimension algorithms}.
As building blocks for changing dimension algorithms we allow general unbiased algorithms
for multivariate integration.
The ``typical algorithms'' we have in mind are (suitably) randomized quasi-Monte Carlo
(RQMC) algorithms, as discussed in Section \ref{SPLR}.
In contrast to \cite{HMNR10, BG12} we investigate in this paper a different, more generous cost model for
algorithms which was introduced in \cite{KSWW10a}.

The paper is organized as follows: In Section \ref{Preliminaries} we recall the
ANOVA decomposition of square integrable functions and introduce the weights,
function spaces, cost and error criteria we want to study.
Additionally, we provide in Section \ref{PROJECTIONS}
new lemmas which are important for our error analysis of randomized algorithms,
particularly for our randomized changing dimension algorithms.
In Section \ref{LOWBOU} we provide a lower bound for the error of infinite-dimensional
integration
of general randomized algorithms and general weights, see Theorem \ref{Theorem3.1''}.
As shown in Section \ref{CDA}, this bound is sharp for finite-intersection and
product weights. More precisely, we present in Section \ref{CDA} randomized
changing dimension algorithms for general
weights and prove an upper bound for their randomized worst-case integration error, see
Theorem \ref{Theo_UB_General}. In Theorem \ref{Theo_UB_FIW} and \ref{Theo_UB_PW} we
provide sharp upper error bounds for finite-intersection weights and
product weights, respectively.
For the analysis of our new changing dimension algorithms we adapt the approach
of Plaskota and Wasilkowski from \cite{PW11}.
In Section \ref{SPLR} we consider concrete spaces of functions of
infinitely many variables. The function spaces we consider are unanchored Sobolev space of smoothness $\chi \ge 1$. By showing that the results of \cite{GD12} apply to these spaces, we obtain that \emph{interlaced scrambled polynomial lattice rules} achieve the optimal rate of convergence of the random case error in these spaces. Based on these results we show that changing dimension algorithms based on interlaced scrambled polynomial lattice rules are essentially optimal in the case of finite-intersection weights and product weights.

\section{Preliminaries} \label{Preliminaries}

Let us make a few remarks on our notation: For $n\in\N$ we denote by $[n]$ the
set $\{1,2,\ldots,n\}$. For a finite set $u$ we denote its cardinality by
$|u|$. We put $\U := \{ u\subset \N \,|\, |u| < \infty\}$. For a subset $\mathcal{W}$ of $\U$ we put
\begin{equation*}
\overline{\mathcal{W}} := \{w\in\U \,|\, \exists w' \in \mathcal{W} : w \subseteq w'\},
\end{equation*}
i.e., $\overline{\mathcal{W}}$ is the closure
of $\mathcal{W}$ with respect to taking subsets. We use the common Landau $o$- and $O$-notation. For functions $f$ and $g$
we write $f=\Omega(g)$ for $g= O(f)$, and $f=\Theta(g)$ if $f=\Omega(g)$ and $f=O(g)$ holds.
For a reproducing kernel $K$ we denote the
corresponding reproducing kernel Hilbert space by $H(K)$ and its norm unit
ball by $B(K)$. The norm and scalar product of $H(K)$ are denoted by
$\|\cdot\|_K$ and $\langle \cdot, \cdot \rangle_K$, respectively.
Our reference for reproducing kernel Hilbert spaces
is \cite{Aro50}.

\subsection{The ANOVA decomposition}\label{SECT2.1}

We recall the (crossed) ANOVA\footnote{ANOVA stands
for ``Analysis of Variance''.} decomposition of $L^2$-functions:
Let {$(D, \Sigma, \rho)$ be a probability space, and denote its $d$-fold product space by $(D^d, \Sigma^d, \rho^d)$.
Let $f:D^d \to \R$ be an $L^2$-function. For $u\subseteq [d]$ and $\bsx\in D^d$ let
$\bsx_u := (x_j)_{j\in u} \in D^u$. For $\bsx_u\in D^u$ and $\bsa \in D^{[d]\setminus u}$ let $(\bsx_u, \bsa)\in D^d$
be the vector
whose $j$th component is $x_j$ if $j\in u$ and $a_j$ otherwise. The $u$th ANOVA-term $f_u$ of $f$ can be
computed recursively via
\begin{displaymath}
f_u(\bsx) = \int_{D^{[d] \setminus u}} f(\bsx_u, \bsa) \,\rho^{[d]\setminus u} ({\rm d} \bsa)
- \sum_{v \subsetneq u} f_v(\bsx) \, , \textrm{ where } f_{\emptyset} = \int_{D^d} f(\bsa) \,
\rho^d({\rm d} \bsa).
\end{displaymath}
The ANOVA decomposition of $f$ is given by
\begin{equation} \label{eqANOVAdecomp}
f(\bsx) = \sum_{u \subseteq [d]} f_u(\bsx) \, .
\end{equation}
The important feature of the ANOVA decomposition is
\begin{equation}
 \label{anova}
\Var(f) = \sum_{u\subseteq [d]} \Var(f_u).
\end{equation}

Let $(\Omega,\Sigma',\PP)$  be another probability space. Let us consider random quadratures
that use $n$ (deterministic) real coefficients $t_i$ and $n$ randomly chosen quadrature points
$\bsx^{(1)}(\omega),\ldots, \bsx^{(n)}(\omega)$ in $D^d$, i.e., that have the form
\begin{equation}
\label{alg-form}
Q_n(\omega, f) = \sum^n_{i=1} t_i f(\bsx^{(i)}(\omega)),
\hspace{3ex}\omega\in\Omega^d,\, t_1,\ldots, t_n \in\R, \,f\in L^2(D^d, \rho^d).
\end{equation}
We assume that for every fixed $f\in L^2(D^d, \rho^d)$  the function $\omega \mapsto Q_n(\omega,f)$ is square integrable over $\Omega^d$.

For the convenience of the reader we provide a (less general) version of Lemma~2.1 from \cite{BG12},
which says that under a certain condition the $u$th ANOVA-term
of the $L^2(\Omega^d,\PP^d)$-function  $Q_n(\cdot,f)$ is equal to $Q_n$
applied to the $u$th ANOVA-term of the $L^2(D^d, \rho^d)$-function $f$.
We denote the ANOVA-terms of $Q_n(\cdot,f)$, regarded as a function on $\Omega^d$, by $\left[ Q_n(\cdot, f) \right]_u$, $u\subseteq [d]$.

\begin{lemma}[ANOVA Invariance Lemma]
\label{ANOVA}
Let $(D, \Sigma, \rho)$, $(\Omega,\Sigma', \PP)$ be probability spa\-ces.
Let $d\in\N$.  Assume that $Q_n = Q_{[d],n}$, given by (\ref{alg-form}),
is a randomized linear algorithm which
satisfies the following condition:
\begin{itemize}
\item[(*)] The random points $\bsx^{(i)} = (x^{(i)}_j(\omega))_{j=1}^d \in D^d$,
$i=1,\ldots,n$,  satisfy $x^{(i)}_j(\omega) = x^{(i)}_j(\omega_j)$ for all $j\in [d]$,
and the random variables $x^{(i)}_j$ are distributed according to the law $\rho$.
\end{itemize}
Then we have for each $f\in L^2(D^d, \rho^d)$
\begin{equation}
\label{vertauschungsrelation}
[Q_n(\cdot, f)]_u = Q_n(\cdot, f_u)
\hspace{3ex}\text{for all $u\subseteq [d]$.}
\end{equation}
\end{lemma}

Note that for $f\in L^2(D^d,\rho^d)$ condition (*) of Lemma \ref{ANOVA} implies that
$Q_n(\cdot,f)$ is square integrable on $\Omega^d$ and, if additionally
\begin{equation}\label{summe=1}
\sum_{i=1}^n t_i = 1
\end{equation}
holds, an unbiased estimator of
$\int_{D^d}f(\bsx) \,\rho^d({\rm d}\bsx)$.

Let us point out that Monte Carlo (MC) and many randomized quasi-Monte Carlo (RQMC) algorithms
satisfy condition (*) of Lemma \ref{ANOVA}, see, e.g., \cite{DKS13, DP10, LL02}. We will
demonstrate this for scrambled polynomial lattice rules in Section~\ref{SPLR}.

\subsection{Classes of weights}

Let
\begin{equation*}
\U := \{u\subset \nn \,|\, |u|<\infty\},
\end{equation*}
and let $\bsgamma=(\gamma_u)_{u\in \U}$ be a sequence of non-negative weights.
Let us briefly introduce the classes of weights we are interested in.

Weights $\bsgamma$ are called \emph{finite-order weights of order $\beta$} if  there exists a $\beta \in \N$ such that $\gamma_{u}=0$ for all $u\in \U$ with $|u|> \beta$. Finite-order weights were introduced in \cite{DSWW06} for spaces of functions with a finite number of variables.


\emph{Product and order-dependent (POD) weights} $\bsgamma$ were introduced in \cite{KSS12}.
Their general form is
\begin{equation}
\label{pod}
\gamma_u = \Gamma_{|u|} \prod_{j\in u}\gamma_j,
\hspace{3ex}\text{where $\gamma_1\ge\gamma_2\ge \cdots \ge 0$, and
$\Gamma_0=\Gamma_1=1$, $\Gamma_2,\Gamma_3,\ldots \ge 0$.}
\end{equation}
Special cases are product and finite-product weights that are defined as
follows.

\begin{definition}
Let $(\gamma_j)_{j\in\N}$ be a sequence of non-negative real
numbers satisfying $\gamma_1\ge \gamma_2 \ge \ldots.$ With the
help of this sequence we define for $\beta \in\N\cup\{\infty\}$  weights
$\bsgamma = (\gamma_{u})_{u\subset_f\N}$ by
\begin{equation}
\label{gammafpw}
\gamma_{u} =
\begin{cases}
\prod_{j\in u} \gamma_j
\hspace{2ex}&\text{if $|u| \le \beta$},\\
\,0
\hspace{2ex} &\text{otherwise,}
\end{cases}
\end{equation}
where we use the convention that the empty product is $1$.
In the case where $\beta = \infty$, we call such weights \emph{product weights},
in the case where $\beta$ is finite, we call them \emph{finite-product weights of order}
(at most) $\beta$.
\end{definition}

Product weights were introduced by  Sloan and Wo\'zniakowski in \cite{SW98}, finite-product weights were considered in \cite{Gne10}. We are particularly interested in some subclass of finite-order weights. We restate Definition 3.5 from \cite{Gne10}.

\begin{definition}
\label{def-fiw}
Let $\varrho \in \nn$.
Finite-order weights $(\gamma_{u})_{u\in \U}$ are
called \emph{finite-intersection weights} with \emph{intersection
degree} at most $\varrho$ if we have
\begin{equation}
\label{fiw}
|\{v\in\U \, | \, \gamma_v >0 \,,\, u\cap v \neq \emptyset \}| \le 1+ \varrho
\hspace{2ex}\text{for all $u\in\U$ with $\gamma_u >0$.}
\end{equation}
\end{definition}
Note that for finite-order weights of order $\beta$, condition (\ref{fiw}) is
equivalent to the following
condition: There exists an $\eta\in\nn$ such that
\begin{equation}\label{cond}
|\{ u\in\U \,|\, \gamma_u >0 \,,\, k\in u \}| \le \eta
\hspace{2ex}\text{for all $k\in\nn$.}
\end{equation}
A subclass of the finite-intersection weights are the
\emph{finite-diameter weights} proposed by Creutzig, see, e.g., \cite{Gne10,NW08}.

\subsection{Function Spaces} \label{subsecfunspace}

Let $D\subseteq \rr$, $\rho$ a probability measure on $D$, and
$\mu:=\otimes_{n\in\nn}\, \rho$ the product probability measure on
$D^\nn$.
Let $(\gamma_u)_{u\in \U}$ be a family of
non-negative weights.


\begin{assumptions}
 We assume that
\begin{itemize}
\item[{\rm (A 1)}] $k\neq 0$ is a measurable reproducing kernel on $D\times D$
which satisfies
\item[{\rm (A 2)}] $H(k)\cap H(1) = \{0\}$ as well as
\item[{\rm (A 3)}] $M:= \int_D k(x,x) \rho({\rm d}x) < \infty$.
\item[{\rm (A 4)}] $\gamma_\emptyset = 1$ and
\begin{equation}
\label{summable}
\sum_{u\in \U} \gamma_u M^{|u|} <\infty.
\end{equation}
\end{itemize}
\end{assumptions}
Notice that for product weights and finite-order weights condition (\ref{summable}) can be replaced by the equivalent condition
$\sum_{u\in \U} \gamma_u < \infty$.

For $u\in \U$ we put $k_u(\bsx,\bsy) := \prod_{j\in u} k(x_j,y_j)$, for all $\bsx, \bsy\in D^\nn$.
In particular, $k_\emptyset(\bsx,\bsy) = 1$.
We define $H_u := H(k_u)$, i.e., $H_u$ is the reproducing kernel
Hilbert space with kernel $k_u$. The following lemma stems from \cite{HMNR10}.
\begin{lemma}
\label{restrict}
Let $\bsx,\bsy\in D^{\N}$ and $f\in H_u$. If $\bsx_u = \bsy_u$,
then $f(\bsx) = f(\bsy)$.
\end{lemma}

Given $v\in\mathcal{U}$ we define the \emph{weighted kernel} $K_v(\bsx,\bsy) := \sum_{u\subseteq v} \gamma_u k_u(\bsx,\bsy)$, for $\bsx,\bsy\in D^{\nn}$.
For the next lemma see \cite[Lemma~3]{HW01} or \cite[I, \S~6]{Aro50}.
\begin{lemma}\label{Lemma5}
The reproducing kernel Hilbert space $H(K_v)$ consists of all
functions $f = \sum_{u\subseteq v} f_u$, for $f_u\in H_u$.
Furthermore, $\|f\|^2_{K_v} = \sum_{u\subseteq v} \gamma^{-1}_u \|f_u\|^2_{k_u}.$
\end{lemma}

We follow here and elsewhere the convention that $\infty\cdot 0 = 0$. Note that
$\gamma_u = 0$ implies $f_u=0$ for all $f\in H(K_v)$;  in that
case  $\gamma^{-1}_u\|f_u\|^2_{k_u} = 0$.

Due to Lemma \ref{restrict} we can view the spaces $H(k_u)$
and $H(K_u)$ as spaces of functions on $D^u$.
In this case we have
$H(k_u) = \otimes_{j\in u} H(k)$.

Let us define the domain $\X$ of functions of infinitely many variables by
\begin{equation}
 \X := \bigg\{ \bsx\in D^{\N} \,\bigg|\, \sum_{u\in \U} \gamma_u \prod_{j\in u} k(x_j, x_j) <\infty \bigg\}.
\end{equation}
Then $\mu(\X) = 1$, see \cite[Lemma~9]{GMR12}.
We define the reproducing kernel $K=K(\bsgamma)$ by
\begin{equation*}
K(\bsx,\bsy) := \sum_{u\in\mathcal{U}} \gamma_u k_u(\bsx,\bsy)
\hspace{2ex}\text{for $\bsx,\bsy \in\X$.}
\end{equation*}
For the next lemma see
\cite[Cor.~5]{HW01} or
\cite[Prop.~2]{GMR12}.

\begin{lemma}
\label{Lemma6}
The reproducing kernel Hilbert space $H(K)$ consists of all functions
\begin{equation}\label{function-decomp}
f=\sum_{u\in \mathcal{U}} f_u,
\hspace{3ex}f_u \in H_u,
\end{equation}
whose norms $\|f \|_K$, defined by
\begin{equation}\label{norm-formel}
\|f\|^2_K = \sum_{u\in\mathcal{U}} \gamma^{-1}_u \|f_u\|^2_{k_u},
\end{equation}
are finite.
\end{lemma}

If $f\in H(K)$, then the decomposition (\ref{function-decomp})
is uniquely defined, since $f_u$ is the orthogonal
projection of $f$ onto $H_u$.

\subsection{Integration}

It is easily verified with the help of the reproducing property, the Cauchy-Schwarz
inequality, and (\ref{summable}) that $H(K) \subset L^2(\X,\mu)$ and $H_u \subset
L^2(D^u, \rho^u)$ for all $u\in \U$. This implies in particular that
integration with respect to the probability measure $\mu$
defines a bounded linear functional
\begin{equation*}
I(f) := \int_{\X} f(\bsx)\, \mu({\rm d}\bsx)
\end{equation*}
on $H(K)$.
The representer $h\in H(K)$ of the integration functional $I$ is given by
\begin{equation}
\label{representer}
h(\bsx) = \langle h, K(\cdot, \bsx) \rangle_K
= \int_{\X} K(\bsx,\bsy)\, \mu({\rm d}\bsy).
\end{equation}
Similarly, for every $u\in H_u$
\begin{equation*}
I_u(f) := \int_{D^u} f(\bsx)\, \rho^u({\rm d}\bsx)
\end{equation*}
defines a bounded linear functional on $H_u$. For the rest of this article we assume that the following
assumptions hold:
\begin{assumptions}
We assume that
\begin{itemize}
 \item[{\rm (A~2a)}]
$\int_D k(x,y)\,\rho({\rm d}x) = 0$
for all $y\in D$.
 \item[{\rm (A~5)}]
For all $a\in D$ we have $k(a,a)>0$.
\item[{\rm (A~6)}]
If $\gamma_v>0$ for $v\in\U$, then $\gamma_u>0$ for all $u\subseteq v$.
\end{itemize}
\end{assumptions}
Note that identity (\ref{representer}), $\gamma_{\emptyset} =1$, and assumption (A~2a)
immediately imply that
\begin{equation}
\label{h=1}
h(\bsx) = 1 \hspace{2ex}\text{for all $\bsx\in \X$.}
\end{equation}
Thus, if there exists an $a^*\in D$ with $k(a^*,a^*) = 0$, then this
results for $\bsa^*:=(a^*)_{j\in\N}$ in $K(\cdot,\bsa^*) = h$,
which leads to $I(f) = f(\bsa^*)$ for all $f\in H(K)$.
Assumption (A~5) avoids this trivial integration problem.

Under assumption (A~2a), the uniquely determined decomposition (\ref{function-decomp})
is in fact the ANOVA decomposition of $f$, see \cite[Remark~2.2]{BG12}.

\subsection{Projections}\label{PROJECTIONS}

Let us choose an anchor $\bsa\in \X$. A natural choice are vectors $\bsa$ whose entries are all equal to $a$,
where $a\in D$ satisfies
\begin{equation}
 \label{anker_a}
\sum_{u\in \U} \gamma_u k(a,a)^{|u|} < \infty;
\end{equation}
note that (\ref{summable}) ensures that such an $a$ exists.
Note that it is possible to consider a general $\bsa \in \X$, but to make proofs not unnecessarily
complicated, we will restrict
ourselves to anchors $\bsa = (a,a,\ldots)$ for the concrete
analysis of our constructive changing dimension algorithms in the case of product weights and
finite-intersection weights.
We define for $u\in \mathcal{U}$
\begin{equation*}
(\Psi_{u,\bsa}f)(\bsx) := f(\bsx_u; \bsa)
\hspace{2ex}\text{for all $\bsx\in \X$,}
\end{equation*}
where $(\bsx_u;\bsa) := (\bsx_u, \bsa_{\N\setminus u})$, i.e., the $j$th entry of this infinite-dimensional
vector is $x_j$ if $j\in u$ and $a_j$ otherwise.
Note that due to $\bsa \in \X$ we have $(\bsx_u;\bsa) \in \X$.
For $u,v \in \U$ with $u\subseteq v$ we put
\begin{equation}
\label{r}
r^2_{v,u,\bsa} :=
\sum_{u'\subset \N\setminus v}
\gamma_{u\cup u'}\, k_{u'}(\bsa,\bsa).
\end{equation}
Due to $\bsa \in \X$ and assumption (A~5), the quantity $r^2_{v,u,\bsa}$ is finite.
Due to assumption (A~6) the mapping
{$\Psi_{v,\bsa}$ is a bounded
projection from $H(K)$ onto $H(K_v)$, and its operator norm is given by
\begin{equation}\label{op_norm_pro}
 \|\Psi_{v,\bsa}\|_{K\to K} = \max_{u\subseteq v} \gamma^{-1/2}_u r_{v,u,\bsa},
\end{equation}
see \cite[Lemma~2.7]{BG12}.
For $u, v \in\mathcal{U}$ with $u\subseteq v$ we define
\begin{equation}
\label{f+-}
f_{u,v}^+ := \sum_{u' \subset \N\setminus v} f_{u\cup u'}
\hspace{2ex}\text{and}\hspace{2ex}
f^+_u := f^+_{u,u}.
\end{equation}
Remark~2.2 and Lemma~2.7 from \cite{BG12} result in the following lemma.

\begin{lemma}\label{Prop}
Let $f\in H(K)$, and let $v\in\U$. Then we have the orthogonal decomposition
\begin{equation}\label{projection}
f = \sum_{u\subseteq v} f_{u,v}^+,
\end{equation}
and the $u$th ANOVA component of $\Psi_{v,\bsa}(f)$ is given by
\begin{equation}\label{bg12}
[\Psi_{v,\bsa}(f)]_{u} =
\begin{cases}
\Psi_{v,\bsa}(f^+_{u,v})
\hspace{2ex}&\text{if $u \subseteq v$},\\
\,0
\hspace{2ex} &\text{otherwise.}
\end{cases}
\end{equation}
\end{lemma}

For $u\in \U$ let $f_{u,\bsa}$ denote the $u$th component of the \emph{anchored decomposition} of $f$
with respect to the anchor $\bsa$, i.e.,
\begin{equation}\label{anchored_decomposition}
 f_{u,\bsa} := \Psi_{u,\bsa}(f) - \sum_{v\subsetneq u} f_{v,\bsa}.
\end{equation}
Then
\begin{equation}\label{ksww10b}
 f_{u,\bsa} = \sum_{v\subseteq u} (-1)^{|u\setminus v|} \Psi_{v,\bsa}(f),
\end{equation}
see \cite[Example~2.3]{KSWW10b}. Note that for $u,w\in \U$ with $u\nsubseteq w$ we have $\Psi_{w,\bsa} (f_{u,\bsa}) = 0.$
Due to (\ref{projection}) we obtain the additional representation
\begin{equation*}
 f_{u,\bsa}
= \sum_{w\subseteq u} \sum_{w\subseteq v \subseteq u} (-1)^{|u\setminus v|} \Psi_{v,\bsa}(f_{w,v}^+).
\end{equation*}
Hence the $w$th ANOVA component of $f_{u,\bsa}$ is given by
\begin{equation}\label{wth_ANOVA_comp}
 [f_{u,\bsa}]_w = \sum_{w\subseteq v \subseteq u} (-1)^{|u\setminus v|} \Psi_{v,\bsa}(f_{w,v}^+).
\end{equation}

As shown in the next lemma, this representation can be simplified.

\begin{lemma}\label{Umrechnungsformel}
 Let $f\in H(K)$, and
let $u, w \in \U$.
Then the $w$th ANOVA component of the anchored component $f_{u,\bsa}$ of $f$ has the form
\begin{equation}\label{anchored_ANOVA}
[f_{u,\bsa}]_w =
\begin{cases}
(-1)^{|u\setminus w|} \Psi_{w,\bsa}(f^+_{u})
\hspace{2ex}&\text{if $w \subseteq u$},\\
\,0
\hspace{2ex} &\text{otherwise.}
\end{cases}
\end{equation}
\end{lemma}

\begin{proof}
We have $[f_{u,\bsa}]_w = 0$ if $w \nsubseteq u$ (this follows, e.g., from (\ref{wth_ANOVA_comp})).
Thus let now $w\subseteq u$.
 We prove the following more general statement via induction on $|w'|$: For each $w' \subseteq u\setminus w$ we have
\begin{equation}\label{**}
[f_{u,\bsa}]_w =  \sum_{w \subseteq v \subseteq u\setminus w'} (-1)^{|u\setminus v|} \Psi_{v,\bsa}(f_{w\cup w', v\cup w'}^+).
\end{equation}
Notice that the special case $w' = u\setminus w$ of identity (\ref{**}) is precisely identity (\ref{anchored_ANOVA}).

Let us start with $|w'|=0$, i.e., $w'=\emptyset$. Then (\ref{**}) holds due to identity (\ref{wth_ANOVA_comp}).

So let now $|w'|\ge 1$, and let us assume that (\ref{**}) holds for all $w'' \subseteq u\setminus w$ with $|w''| < |w'|$.
Let $i\in w'$, and set $w'' := w'\setminus \{i\}$.
Due to our induction hypothesis we obtain
\begin{equation*}
 \begin{split}
[f_{u,\bsa}]_w = &\sum_{w\subseteq v \subseteq u\setminus w''} (-1)^{|u\setminus v|} \Psi_{v,\bsa}(f_{w\cup w'', v\cup w''}^+)\\
= &\sum_{w\subseteq v \subseteq u\setminus w'} (-1)^{|u\setminus v|} \big( \Psi_{v,\bsa}(f_{w\cup w'', v\cup w''}^+)
  - \Psi_{v\cup \{i\},\bsa}(f_{w\cup w'', v\cup w'}^+) \big).
 \end{split}
\end{equation*}
Since $i\notin w \cup w''$, the function $f^+_{w\cup w'', v\cup w'}$ does not depend on the $i$th variable, implying
\begin{equation*}
 \Psi_{v\cup \{i\},\bsa}(f_{w\cup w'', v\cup w'}^+) = \Psi_{v,\bsa}(f_{w\cup w'', v\cup w'}^+)
\end{equation*}
(cf. also \cite[Lemma~2]{HMNR10} or see \cite[Lemma~2.4]{BG12}).
Furthermore,
\begin{equation*}
\begin{split}
f^+_{w\cup w'', v\cup w''} - f^+_{w\cup w'', v\cup w'}
= &\sum_{i \in u'\subset \N\setminus (v\cup w'')} f_{w\cup w''\cup u'} \\
= &\sum_{u''\subset \N\setminus (v\cup w')} f_{w\cup w'\cup u''}
= f^+_{w\cup w', v\cup w'}.
\end{split}
\end{equation*}
This leads to
\begin{equation*}
[f_{u,\bsa}]_w
= \sum_{w\subseteq v \subset u\setminus w'} (-1)^{|u\setminus v|}  \Psi_{v,\bsa}(f_{w\cup w', v\cup w'}^+).
\end{equation*}
This shows that (\ref{**}) is valid for all $w' \subseteq u\setminus w$, which implies in particular (\ref{wth_ANOVA_comp}).
\end{proof}

The next lemma is essential for our upper error bounds in Section~\ref{CDA}.

\begin{lemma}
\label{Analogon-Lemma7}
For all $f\in H(K)$ and all finite subsets
$w\subseteq u \subseteq v$ of $\N$ we have
$\Psi_{w,\bsa}(f^+_{u,v}) \in H_w$ and
the norm estimate
\begin{equation}
\label{norm_psi_k}
\| \Psi_{w,\bsa}
(f^+_{u,v}) \|_{k_{w}}
\le r_{v,u,\bsa} \sqrt{k_{u\setminus w}(\bsa, \bsa)} \|  f^+_{u,v} \|_{K}.
\end{equation}
\end{lemma}

\begin{proof}
We reduce the lemma to \cite[Lemma~2.7]{BG12}.  Put $g:= f^+_{u,v}$. Then
\begin{equation*}
 g^+_{w} = \sum_{w'\,;\, w\subseteq w'} g_{w'}
= \sum_{w'\,;\, w\subseteq w'} [f^+_{u,v}]_{w'}
=f^+_{u,v} = g.
\end{equation*}
Let us define the auxiliary weights $\overline{\bsgamma} = (\overline{\gamma}_{u'})_{u'\in \U}$,
by $\overline{\gamma}_{u'} = \gamma_{u'}$ if $u'\cap v =u$, and $\overline{\gamma}_{u'} = 0$ otherwise,
and let $\overline{K} = K(\overline{\bsgamma})$ denote the corresponding kernel.
Then $g=g^+_{w} \in H(\overline{K})$, and we get from \cite[Lemma~2.7]{BG12}
that $\Psi_{w,\bsa}(f^+_{u,v}) = \Psi_{w,\bsa}(g^+_{w}) \in H_w$, and
\begin{equation*}
\|\Psi_{w,\bsa}(f^+_{u,v}) \|_{k_w} =
\|\Psi_{w,\bsa}(g^+_{w}) \|_{k_w}
\le \overline{r}_{w,w,\bsa} \|g^+_{w}\|_{\overline{K}}
= \overline{r}_{w,w,\bsa} \|f^+_{u,v}\|_{K}\,,
\end{equation*}
where
\begin{equation*}
\begin{split}
\overline{r}_{w,w,\bsa}^2 = &\sum_{w'\subset \N\setminus w} \overline{\gamma}_{w\cup w'} k_{w'}(\bsa,\bsa)
= \sum_{w'\subset \N\setminus w \atop w' \cap v = u\setminus w} \gamma_{w\cup w'} k_{w'}(\bsa,\bsa)\\
= &\left( \sum_{u'\subset \N\setminus v} \gamma_{u\cup u'} k_{u'}(\bsa,\bsa) \right) k_{u\setminus w}(\bsa, \bsa)
= r_{v,u,\bsa}^2 k_{u\setminus w}(\bsa,\bsa).
\end{split}
\end{equation*}
This concludes the proof.
\end{proof}

For $v_1,\ldots, v_n \in \mathcal{U}$ we use the short hand $\{v_i\}$ for $(v_i)^n_{i=1}$.
Define the mapping
\begin{equation}
\label{psi-decomp}
\Psi_{\{v_i\},\bsa}(f) := \sum_{u;\, \exists i\in [n]: u \subseteq v_i} f_{u,\bsa}
\hspace{2ex}\text{for all $f\in H(K)$.}
\end{equation}
The operator $\Psi_{\{v_i\},\bsa}$ is a continuous projection
that maps  $H(K)$ into $H(K_{\{v_i\}})$, where
the kernel $K_{\{v_i\}}$ is defined by
\begin{equation*}
 K_{\{v_i\}}(\bsx, \bsy) := \sum_{u;\, \exists i\in [n]: u \subseteq v_i} \gamma_u k_u(\bsx,\bsy)
\hspace{2ex}\text{for all $\bsx,\bsy\in\X$.}
\end{equation*}

\begin{lemma}\label{Rep_Lemma}
The functional $I\circ \Psi_{\{v_i\},\bsa}$ is continuous on $H(K)$, and its
representer $h_{\{v_i\}, \bsa}$ is given by
\begin{equation}\label{formel_rep}
  h_{\{v_i\}, \bsa}(\bsx) = \sum_{u\in \U} S_{\{v_i\}, u} \gamma_u k_u(\bsx, \bsa)
\hspace{2ex}\text{for every $\bsx \in \X$,}
\end{equation}
where
\begin{equation}\label{def_alt_sum}
 S_{\{v_i\}, u} := \sum_{u';\, \exists i\in [n]: u' \subseteq v_i \cap u} (-1)^{|u'|}
\hspace{2ex}\text{for any $u \in \U$.}
\end{equation}
\end{lemma}

\begin{proof}
For $v\in \U$ the representer $h_{v,\bsa}$ of the continuous functional  $I\circ \Psi_{v,\bsa}$
is given by
\begin{equation*}
 h_{v,\bsa}(\bsx) = \sum_{u\subseteq \N\setminus v} \gamma_u k_u(\bsx, \bsa),
\end{equation*}
see, e.g., \cite{BG12}, proof of Lemma~2.9.
Hence we get
\begin{equation*}
 \begin{split}
  h_{\{v_i\}, \bsa}(\bsx) =  &\sum_{u';\, \exists i\in [n]: u' \subseteq v_i} \sum_{v\subseteq u'} (-1)^{|u'\setminus v|} h_{v,\bsa}(\bsx)\\
= &\sum_{u';\, \exists i\in [n]: u' \subseteq v_i} \sum_{v\subseteq u'} (-1)^{|u'\setminus v|}  \sum_{u\subset \N\setminus v} \gamma_u k_u(\bsx, \bsa)\\
= &\sum_{u\in \U} \left(  \sum_{u';\, \exists i\in [n]: u' \subseteq v_i} \sum_{v\subseteq u'\setminus u} (-1)^{|u'\setminus v|} \right)
\gamma_u k_u(\bsx, \bsa),
\end{split}
\end{equation*}
and the sum $\sum_{v\subseteq u'\setminus u} (-1)^{|u'\setminus v|}$ is equal to $(-1)^{|u'|}$ if $u'\subseteq u$ and, due to the binomial theorem, $0$ otherwise.
This establishes (\ref{formel_rep}).
\end{proof}

\begin{lemma}
\label{Bias-Estimate}
For any sets $v_1, \ldots, v_n \in \mathcal{U}$ we have
\begin{equation*}
{\rm b}_{\{v_i\},\bsa} := \sup_{f\in B(K)} |I(f) - I(\Psi_{\{v_i\}, \bsa}f)| = \sum_{\emptyset \neq u\in \U} S_{\{v_i\}, u}^2
\gamma_u k_u(\bsa,\bsa).
\end{equation*}
\end{lemma}

\begin{proof}
Since the representer of the integration functional $I$ is $h=1$, we obtain from Lemma \ref{Rep_Lemma}
\begin{equation*}
\begin{split}
{\rm b}^2_{\{v_i\},\bsa}
= &\|h-h_{\{v_i\},\bsa}\|^2_K
= \bigg\| \sum_{\emptyset \neq u\in\U} S_{\{v_i\}, u} \gamma_u k_u(\cdot, \bsa) \bigg\|^2_K\\
= &\sum_{\emptyset \neq u \in\U} S_{\{v_i\}, u}^2 \gamma_u \left\|k_u(\cdot, \bsa) \right\|^2_{k_u}
= \sum_{\emptyset \neq u\in\U} S_{\{v_i\}, u}^2\gamma_u k_u(\bsa, \bsa).
\end{split}
\end{equation*}
\end{proof}

\subsection{Randomized algorithms, cost, and error}

We assume that algorithms for approximation of
$I(f)$ have access to the function $f$ via a subroutine (``oracle'')
that provides values $f({\bsx})$ for points ${\bsx}\in D^{\mathbb{N}}$.
For convenience we define $f({\bsx}) = 0$ for ${\bsx}\in D^{\mathbb{N}}\setminus \X$.

We now present the cost models introduced in \cite{KSWW10a}, which we want to call
\emph{unrestricted subspace sampling model} (cf. \cite{DG12, Gne12a}). It only accounts for the cost of function
evaluations. To define the cost of a function evaluation, we fix an anchor $\bsa \in \X$ and a
monotone increasing function $\$: \mathbb{N}_0 \to [1,\infty]$.
For each $v\in\mathcal{U}$ we define the finite-dimensional affine subspace
$\X_{v,\bsa}$ of $\X$ by
\begin{equation*}
 \X_{v,\bsa} := \{ {\bsx} \in D^\mathbb{N} \,|\, x_j = a_j \hspace{1ex}\text{for all}
\hspace{1ex} j\in\mathbb{N}\setminus v\}.
\end{equation*}

In the unrestricted subspace sampling model we are allowed to sample
in any subspace $\X_{u,\bsa}$, $u\in\mathcal{U}$, without any restriction. The
cost for each function evaluation is given by the cost function
\begin{equation}
\label{unrcost}
 c_{\bsa}({\bsx}) := \inf\{ \$(|u|) \,|\, {\bsx}\in \X_{u,\bsa},\, u\in\mathcal{U} \}.
\end{equation}
A different, less generous cost model was introduced
in \cite{CDMR09}. There it was called \emph{variable subspace sampling model} (although the name \emph{nested subspace sampling model}
may be more precise and better to distinguish it clearly from the unrestricted subspace sampling model).
In particular, the articles that build the foundation of our analysis of the randomized ANOVA setting, namely \cite{HMNR10,BG12},
study the variable subspace sampling model and not the unrestricted one.

We consider randomized algorithms for integration of functions $f\in H(K)$.
For a formal definition we refer to \cite{CDMR09,Nov88, TWW88, Was89}.
We require that a randomized algorithm $Q$ yields for each $f\in H(K)$
a square-integrable random variable $Q(f)$. (More precisely, a randomized
algorithm $Q$ is a map $Q:\Omega \times H(K) \to \mathbb{R}$,
$(\omega,f) \mapsto Q(\omega,f)$, where $\Omega$ is some suitable
probability space. But for convenience we will usually not specify the underlying
probability space $\Omega$ and suppress any reference to $\Omega$ or
$\omega\in \Omega$. We use this convention also for other random variables.)
The class of all those randomized algorithms will be denoted by $\mathcal{A}^{{\rm ran}}$.
The \emph{cost} $\cost_{c_{\bsa}}(Q,f)$ of applying a randomized algorithm $Q$ to some function $f$ is simply
the sum of the cost of all function evaluations of $f$ used by $Q$. In general, this cost is a random variable.
We mostly will confine ourselves to randomized algorithms $Q$ for which there exist an $n\in\mathbb{N}_0$ and sets
$v_1,\ldots,v_n\in\mathcal{U}$ such that for every $f\in H(K)$ the algorithm $Q$ performs
exactly $n$ function
evaluations of $f$, where the $i$th sample point is taken from $\X_{v_i,\bsa}$, and
$\mathbb{E}({\rm cost}_{c_{\bsa}}(Q,f)) = \sum^n_{i=1}\$(|v_i|)$.
We denote the class of all randomized algorithms for numerical integration
on $H(K)$ that satisfy the requirements stated above by $\mathcal{A}^{{\rm res}}$ (here ``{\rm res}'' stands for ``restricted''). Notice that the class $\mathcal{A}^{{\rm res}}$ contains in particular non-linear and adaptive algorithms.


The \emph{worst case cost} of a randomized algorithm $Q$ on a class of integrands $F$ is
\begin{equation*}
 {\rm cost}_{}(Q,F) := \sup_{f\in F} \mathbb{E}({\rm cost}_{c_{\bsa}}(Q,f))
\end{equation*}
in the unrestricted subspace sampling model.
The \emph{randomized (worst case) error} $e(Q,F)$ of approximating the integration functional $I$ by $Q$ on $F$ is defined as
\begin{displaymath}
e(Q,F) := \bigg(\sup_{f \in F} \mathbb{E} \left( \left( I(f) - Q(f) \right)^2 \right) \bigg)^{1/2}\, .
\end{displaymath}
For $N\in\mathbb{R}$  let us define the corresponding \emph{$N$th minimal error} by
\begin{equation*}
 e_{}(N,F) := \inf\{ e(Q,F) \,|\, Q\in\mathcal{A}^{{\rm res}}
\hspace{1ex}\text{and}\hspace{1ex}{\rm cost}_{}(Q,F) \le N\}.
\end{equation*}

\subsection{Strong Polynomial tractability}\label{tractabilitysection}
For the convenience of the reader we will additionally formulate our main results in terms of the exponent of
strong tractability.
The  $\varepsilon$-complexity of the infinite-dimensional integration
problem $I$ on $H(K)$ in the unrestricted subspace sampling model with respect to the
class of  randomized algorithms $\mathcal{A}^{\rm res}$ is defined to be
\begin{equation}
\label{comp}
  {\rm comp}_{}(\varepsilon, B(K))
 \,:=\, \inf\left\{{\rm cost_{}}(Q, B(K)) \,|\,
Q \in\mathcal{A}^{{\rm res}} \hspace{1ex}\text{and}
\hspace{1ex} e(Q, B(K))\le\varepsilon\right\}.
\end{equation}
The integration problem $I$ is said to be {\em strongly polynomially
tractable} if there are non-negative constants $C$ and $p$ such that
\begin{equation}
\label{pol-tr}
    {\rm comp}_{}(\varepsilon, B(K))\le C \,\varepsilon^{-p} \qquad
   \mbox{for all $\varepsilon>0$}.
\end{equation}
The {\em exponent of strong polynomial tractability} is given by
\begin{equation*}
p^{{\rm res}}_{}= p^{{\rm res}}_{}({\bsgamma}) := \inf\{ p\,|\, \text{$p$ satisfies \eqref{pol-tr} for some $C>0$} \}.
\end{equation*}
Essentially, $1/p^{{\rm res}}_{}$ is the
\emph{convergence rate} of the
$N$th minimal error
$e_{}(N, B(K))$.
In particular, we have for all $p>p^{{\rm res}}_{}$ that
$e_{}(N, B(K)) = O(N^{-1/p})$.

\section{Lower error bounds}
\label{LOWBOU}

For a fixed anchor $\bsa\in\X$ and weights
$(\gamma_u)_{u\in\mathcal{U}}$ satisfying the assumptions (A4) and (A6)
put
\begin{equation*}
\widehat{\gamma}_u := \gamma_u \, k_u (\bsa,\bsa)
\hspace{2ex}\text{for all $u\in\U$.}
\end{equation*}
Recall that $\bsa \in\X$ ensures that the weights $(\widehat{\gamma}_u)_{u\in\U}$ are summable.
Further, we put
\begin{equation*}
\decay_{\bsgamma} :=
\sup \left\{ p\in \R \,\Big|\,
\sum_{u\in\U} \widehat{\gamma}_{u}^{1/p} <\infty \right\}.
\end{equation*}

\subsection{General weights}

The next two results are helpful for establishing lower bounds
for the randomized error of numerical integration.
The first lemma generalizes \cite[Lemma~3.1]{BG12}.

\begin{lemma}
\label{lemlowbound1} Let $\theta\in (1/2,1]$, $v_1, \ldots, v_n\in\U$, and  let $Q\in \mathcal{A}^{{\rm ran}}$ be a randomized algorithm that satisfies
 $\PP\big( Q(f) = Q(\Psi_{\{v_i\},\bsa}f) \big) \ge \theta$
for all $f\in B(K)$.
Then
\begin{displaymath}
e(Q,B(K)) \ge \max\left\{ \frac{\sqrt{2\theta -1}\, {\rm b}_{\{v_i\},\bsa}}{1+\|\Psi_{\{v_i\},\bsa}\|_{K\to K_{}}} \,, e(Q,B(K_{\{v_i\}})) \right\} \, .
\end{displaymath}
\end{lemma}

\begin{proof}
Put
$\hat{r} := \| \Psi_{\{v_i\},\bsa} \|_{K\to K_{}}$.
Then we have for $f \in B(K)$ that
$g:= (f - \Psi_{\{v_i\},\bsa} (f))(1+\hat{r})^{-1} \in B(K)$.
Furthermore, we have $\Psi_{\{v_i\},\bsa}(g) = \Psi_{\{v_i\},\bsa}(-g) =0$.
Let $A$ denote the event $\{Q(g) = Q(-g) \}$. Then $\PP(A) \ge 2\theta -1$.
Hence
\begin{equation*}
 \begin{split}
&e(Q,B(K))^2 \geq
\max \left\{ \EE \left( \left( I(g) - Q( g) \right)^2 \right) ,
\EE \left( \left(  I(-g) - Q(-g)  \right)^2 \right) \right\} \\
\geq
&\max \left\{ \int_A \left( I(g) - Q( g) \right)^2
\,\PP({\rm d}\omega),
\int_A \left(  I(-g) - Q(-g) ) \right)^2  \,\PP({\rm d}\omega) \right\}\\
\geq &(2\theta -1) \vert I(g) \vert^2 = (2\theta-1)(1+\hat{r})^{-2} \vert I(f) - I(\Psi_{\{v_i\},\bsa}(f)) \vert^2 \, .
\end{split}
\end{equation*}
Since $B \left( K_{\{v_i\}} \right) \subseteq B(K)$, we have additionally
$e \left( Q, B(K) \right) \geq e \left( Q, B(K_{\{v_i\}}) \right)$.
\end{proof}

\begin{lemma}\label{Claim1}
 Let $Q\in\mathcal{A}^{{\rm res}}$ that takes for $i=1,\ldots,n$ its $i$th sample point from
$\X_{v_i,\bsa}$. Then
\begin{equation*}
 Q(f) = Q(\Psi_{\{v_i\},\bsa}(f))
\hspace{2ex}\text{for all $f\in H(K)$.}
\end{equation*}
\end{lemma}

\begin{proof}
Let $j\in [n]$ and $\bsx = (x_{v_j};\bsa)\in \X_{v_j,\bsa}$. Then we have for $f\in H(K)$
\begin{equation*}
 \begin{split}
(\Psi_{\{v_i\}, \bsa}(f))(\bsx) &= \sum_{u\,;\, \exists i: u\subseteq v_i} f_{u,\bsa}(\bsx)
= \sum_{u \subseteq v_j} f_{u,\bsa}(\bsx)\\
&= \sum_{u\subseteq v_j} \sum_{v\subseteq u} (-1)^{|u\setminus v|}
(\Psi_{v,\bsa}(f))(\bsx)
= \sum_{u\subseteq v_j} \sum_{v\subseteq u} (-1)^{|u\setminus v|} f(\bsx_v;\bsa)\\
&=  \sum_{v\subseteq v_j} \left( \sum_{u\,;\,v\subseteq u \subseteq v_j} (-1)^{|u\setminus v|} \right) f(\bsx_v;\bsa).
 \end{split}
\end{equation*}
Notice that the sum in parentheses is one if $v=v_j$ and zero otherwise. Hence
\begin{equation*}
(\Psi_{\{v_i\}, \bsa}(f))(\bsx) = f(\bsx_{v_j};\bsa) = f(\bsx).
\end{equation*}
This shows that the algorithm $Q$ receives the same information for both inputs $f$ and $\Psi_{\{v_i\}, \bsa}(f)$
leading to $Q(f) = Q(\Psi_{\{v_i\}, \bsa}(f))$.
\end{proof}

We provide now a general lower bound for the randomized error
of algorithms from the class $\mathcal{A}^{{\rm res}}$ and general weights.
For weights $\bsgamma$ let us consider the corresponding \emph{cut-off weights
of order $1$}, i.e., the weights $\bsgamma^{(1)} = (\gamma^{(1)}_u)_{u\in\U}$ defined
by $\gamma^{(1)}_u := \gamma_{\{j\}}$ if $u=\{j\}$ and $\gamma_u^{(1)} := 0$ otherwise.
Without loss of generality we may assume that $\gamma^{(1)}_{\{1\}} \ge \gamma^{(1)}_{\{2\}} \ge \cdots$.
Due to the assumption that our integration problem is not trivial and due to
(A~6) we then have $\gamma^{(1)}_{\{1\}} >0$.
Then
$B(K(\bsgamma^{(1)})) \subseteq B(K(\bsgamma))$, and
$e(Q, B(K(\bsgamma^{(1)}))) \le e(Q, B(K(\bsgamma^{})))$ for any randomized
algorithm $Q$.

Furthermore, we assume that there exists an $\alpha >0$ such that for univariate integration
in $H(\gamma^{(1)}_{1}k)$ the $N$th minimal error satisfies
\begin{equation}
\label{ass-1}
 e(N, B(\gamma^{(1)}_{\{1\}}k)) = \Omega(N^{-\alpha/2}).
\end{equation}
Since $B(\gamma^{(1)}_{\{1\}} k) \subseteq B(K)$, it follows that the $N$th minimal
error of integration in $H(K)$ satisfies
\begin{equation}\label{alpha}
 e(N, B(K)) = \Omega(N^{-\alpha/2}).
\end{equation}

\begin{theorem}\label{Theorem3.1''}
Let $\$(\nu) = \Omega(\nu^s)$ for some $s\in (0,\infty)$.
To achieve strong polynomial tractability for the class $\mathcal{A}^{{\rm res}}$
it is necessary that the
weights $\bsgamma$ satisfy
${\rm decay}_{{\bf \gamma}^{(1)}} > 1$.
If this is the case, we have for all $p> \decay_{\bsgamma^{(1)}}$ that
\begin{equation*}
 e(N, B(K))^2 = \Omega \left( N^{-\min\{\alpha, \frac{p-1}{\min\{1,s\}} \}} \right)
\end{equation*}
or, equivalently,
\begin{equation*}
p^{{\rm res}}_{} \ge \max \left\{ \frac{2}{\alpha}\,,\,
\frac{2\min\{1,s\}}{{\rm decay}_{{\bf \gamma}^{(1)}} - 1} \right\}.
\end{equation*}
\end{theorem}

The lower error bound and the lower bound on the exponent of strong tractability
in Theorem \ref{Theorem3.1''} are already optimal for product weights
and for finite-intersection weights, as will be
shown in Section \ref{UB_PW} and \ref{UB_FOW}.

\begin{proof}
Let $Q \in\mathcal{A}^{{\rm res}}$ have
${\rm cost}_{}(Q,B(K)) \le N$.
Then there exists an $n\in\mathbb{N}$ and coordinate sets $v_1,\ldots,v_n$
such that $Q$ selects randomly $n$ sample points ${\bf x}_1\in \X_{v_1,\bsa},
\ldots ,{\bf x}_n \in\X_{v_n,\bsa}$ and $\sum^n_{i=1} \$(|v_i|) \le N$.

To prove a lower bound for $e(Q,B(K(\bsgamma)))$, we actually establish a lower bound for $e(Q,B(K(\bsgamma^{(1)})))$.
Let $v:= \cup_{i=1}^n v_i$. Since clearly $Q(f) = Q(\Psi_{v,\bsa}f)$ for all $f\in B(K(\bsgamma^{(1)}))$, it is straightforward to
deduce with the help of Lemma~\ref{lemlowbound1} and Lemma~\ref{Bias-Estimate} that
\begin{equation*}
 e(Q,B(K({\bf \gamma}^{(1)})))^2 \ge \left[ 1+\| \Psi_{v,\bsa} \|_{K(\bsgamma^{(1)}) \to K(\bsgamma^{(1)})} \right]^{-2}\, \sum_{j \in \N\setminus v} \widehat{\gamma}_{\{j\}} .
\end{equation*}
Due to \eqref{op_norm_pro} we get
\begin{equation*}
 \| \Psi_{v,\bsa} \|_{K(\bsgamma^{(1)}) \to K(\bsgamma^{(1)})}
=  \max\{ r_{v,\emptyset,\bsa}, \max_{j\in v} \gamma^{-1/2}_{\{j\}} r_{v,\{j\},\bsa} \}
= \left( 1 + \sum_{j\in \N\setminus v} \widehat{\gamma}_{\{j\}} \right)^{1/2}.
\end{equation*}
Put
\begin{equation*}
C_{{\rm op}} :=  \left( 1 + \sum_{j \in \N} \widehat{\gamma}_{\{j\}} \right)^{1/2}.
\end{equation*}
This quantity is finite, and $e(Q,B(K(\bsgamma^{(1)})))^2 \ge [1+C_{{\rm op}}]^{-2} \sum_{j\in \N\setminus v} \widehat{\gamma}_{\{j\}}$.
With Jensen's inequality we get with a suitable constant $c >0$
\begin{equation*}
|v| \le \sum^n_{i=1} |v_i| \\
\le \left( \sum^n_{i=1} |v_i|^s \right)^{1/\min\{1,s\}}
\le (cN)^{1/\min\{1,s\}}.
\end{equation*}
Hence we obtain for $S:= \lceil (cN)^{1/\min\{1,s\}} \rceil$
and all $p > {\rm decay}_{{\bf \gamma}^{(1)}}$
\begin{equation}\label{hilfsgleichung}
 e(Q,B({\bf \gamma}^{(1)}))^2 \ge \sum^\infty_{j=S+1} \widehat{\gamma}_{\{j\}}
= \Omega( S^{1-p} )
= \Omega \left( N^{\frac{1-p}{\min\{1,s\}}} \right).
\end{equation}
From this and (\ref{alpha}) the error estimate in Theorem \ref{Theorem3.1''} and
the inequality for the exponent of strong tractability follow.

Now assume that the infinite-dimensional integration problem $I$ is strongly polynomially tractable.
We get from Inequality~\eqref{hilfsgleichung} for all $p > {\rm decay}_{{\bf \gamma}^{(1)}}$
that $p \ge 1 + 2\min\{1,s\}/ p^{{\rm res}}_{}$. Hence
\begin{equation*}
{\rm decay}_{{\bf \gamma}^{(1)}} \ge 1 + \frac{2\min\{1,s\}}{p^{{\rm res}}_{}}.
\end{equation*}
Therefore we have ${\rm decay}_{{\bf \gamma}^{(1)}} >1$.
\end{proof}

\section{Changing dimension algorithms} \label{CDA}

Firstly, we discuss changing dimension algorithms in the ANOVA setting for general weights,
and subsequently show how to tailor them to product weights
and to
finite-intersection weights.

\subsection{General weights} \label{subsecgenweights}

A \emph{changing dimension algorithm} $Q^{\CD}$ is of the
form
\begin{equation}\label{cda}
Q^{\CD}(f) = \sum_{u \in \mathcal{Q}} Q_{u,n_u}(f_{u,\bsa}),
\end{equation}
where, as before,
$f_{u,\bsa}$ is the $u$th component of the anchored decomposition
of $f$ with respect to an anchor $\bsa$, $\mathcal{Q}$ is a finite subset of $\U$, and
$Q_{u,n_u}$ is a quadrature rule using $n_u$ sample points for approximating
$\int_{[0,1]^u} f_{u,\bsa}(\bsx_u) \,{\rm d}\bsx_u$. In particular, we assume that
\begin{equation*}
 \emptyset \in \mathcal{Q} \,,\quad \, n_\emptyset = 1\,,\,
\hspace{2ex}\text{and}\hspace{2ex}
Q_{\emptyset , n_{\emptyset}}(f) = f(\bsa).
\end{equation*}

The algorithm $Q^{\CD}$ is linear and
the cost for evaluating $f_{u,\bsa}$ in the unrestricted subspace sampling model is bounded from above by
$O( 2^{|u|} \$(|u|))$; this follows directly from (\ref{ksww10b}).

Changing dimension algorithms for infinite-dimensional integration in the anchored setting were introduced
by Kuo, Sloan, Wasilkowski, and Wo\'zniakowski in \cite{KSWW10a} and refined by Plaskota and Wasilkowski in
\cite{PW11}. Algorithms for multivariate integration based on a similar
idea were proposed by Griebel and Holtz in \cite{GH10} and referred to as
\emph{dimension-wise quadrature methods}.

If we speak of \emph{randomized changing dimension algorithms} $Q^{\CD}$, then we always assume that the quadratures $Q_{u,n_u}$ are randomized  algorithms and that for each $f\in H(K)$ the random variables $Q_{u, n_u}(f_{u,\bsa})$, $u\in\mathcal{Q}$, are stochastically independent. For our upper bounds we consider quadratures $Q_{u,n_u}$ that are \emph{unbiased} on $L^2(D^u, \rho^u)$.

For convenience, we use for $f\in H(K)$ the notation
\begin{equation*}
 \Psi_{\mathcal{Q},\bsa}(f) := \sum_{u\in\mathcal{Q}} f_{u,\bsa}
\hspace{2ex}\text{and}\hspace{2ex}
{\rm b}_{\mathcal{Q}, \bsa} := \sup_{f\in B(K)} |I(f) - I(\Psi_{\mathcal{Q}, \bsa}f)|.
\end{equation*}
Furthermore, we put
\begin{equation*}
 S_{\mathcal{Q}, u} := \sum_{v \in \mathcal{Q} \,;\, v\subseteq u} (-1)^{|v|}
\hspace{2ex}\text{for any $u \in \U$.}
\end{equation*}

\begin{remark}
 Notice that for $v_1,\ldots,v_n\in \U$ we have $\Psi_{\{v_i\},\bsa} = \Psi_{\mathcal{Q},\bsa}$ if we
put $\mathcal{Q} := \{u\in\U \,|\, \exists i \in [n]: u \subseteq v_i\}$. In this sense the definition
of $\Psi_{\mathcal{Q},\bsa}$ generalizes the one of $\Psi_{\{v_i\},\bsa}$ in Section~\ref{Preliminaries}.
Analogously to Lemma~\ref{Rep_Lemma}, the representer $h_{\mathcal{Q}, \bsa}$ of the continuous functional
$I\circ \Psi_{\mathcal{Q},\bsa}$ in $H(K)$ is given by
\begin{equation}\label{hQa}
h_{\mathcal{Q},\bsa} (\bsx) = \sum_{u\in\U} S_{\mathcal{Q},u} \gamma_u k_u(\bsx,\bsa)
\hspace{2ex}\text{for every $\bsx\in\X$.}
\end{equation}
\end{remark}

\begin{lemma}
 Let $Q^{\CD}$ be a randomized changing dimension algorithm as in (\ref{cda}) with unbiased randomized
quadratures $Q_{u,n_u}$. Then we have for all $f\in H(K)$
\begin{equation}\label{expectation}
\EE(Q^{\CD}(f)) = (I\circ \Psi_{\mathcal{Q},\bsa})(f),
\end{equation}
and the worst case bias ${\rm b}_{\mathcal{Q},\bsa}$ of $Q^{\CD}$ is given by
\begin{equation}\label{bias_cda}
{\rm b}_{\mathcal{Q},\bsa} = \sum_{\emptyset \neq u\in \U} S_{\mathcal{Q}, u}^2
\gamma_u k_u(\bsa,\bsa).
\end{equation}
If additionally the algorithms $Q_{u,n_u}$ satisfy the condition (*) from Lemma \ref{ANOVA}, we have
\begin{equation}\label{variance}
 \Var(Q^{\CD}(f)) = \sum_{\emptyset \neq u\in\mathcal{Q}} \sum_{w\subseteq u} \Var (Q_{u,n_u}(\Psi_{w,\bsa}(f^+_u))),
\end{equation}
implying
\begin{equation}\label{Int_Err}
 \EE \left( I(f) - Q^{\CD}(f) \right)^2
\le {\rm b}_{\mathcal{Q},\bsa}^2 + \sum_{\emptyset \neq u\in\mathcal{Q}} \sum_{w\subseteq u} \Var (Q_{u,n_u}(\Psi_{w,\bsa}(f^+_u))).
\end{equation}
\end{lemma}

\begin{proof}
 We obtain
\begin{equation*}
 \EE(Q^{\CD}(f)) = \sum_{u\in \mathcal{Q}} \EE(Q_{u,n_u}(f_{u,\bsa}))
= \sum_{u\in \mathcal{Q}} \int_{D^u} f_{u,\bsa} \,{\rm d}\rho^u
= \int_{\X} \Psi_{\mathcal{Q},\bsa}(f) \,{\rm d}\mu.
\end{equation*}
With the help of (\ref{hQa}) identity (\ref{bias_cda}) can be proved in the
same way as the identity in Lemma \ref{Bias-Estimate}.
Furthermore, we get we the help of (\ref{anova}), Lemma \ref{ANOVA} and \ref{Umrechnungsformel}
\begin{equation*}
\begin{split}
 \Var(Q^{\CD}(f))
&= \sum_{\emptyset \neq u\in\mathcal{Q}} \Var (Q_{u,n_u}(f_{u,\bsa})
= \sum_{\emptyset \neq u\in\mathcal{Q}} \Var \left( \sum_{w\subseteq u} [Q_{u,n_u}(f_{u,\bsa})]_w \right)\\
&= \sum_{\emptyset \neq u\in\mathcal{Q}}  \sum_{w\subseteq u} \Var \left( Q_{u,n_u}( [f_{u,\bsa}]_w) \right)\\
&= \sum_{\emptyset \neq u\in\mathcal{Q}} \sum_{w\subseteq u} \Var (Q_{u,n_u}((-1)^{|u\setminus w|} \Psi_{w,\bsa}(f^+_u)))\\
&= \sum_{\emptyset \neq u\in\mathcal{Q}} \sum_{w\subseteq u} \Var (Q_{u,n_u}(\Psi_{w,\bsa}(f^+_u))).
\end{split}
\end{equation*}
Let $f\in B(K)$.
Due to (\ref{expectation}) and (\ref{variance}) we see that the bias of $Q^{\CD}(f)$ is given by
\begin{equation*}
 \bias(Q^{\CD},f) = |I(f) - (I\circ \Psi_{\mathcal{Q},\bsa})(f)|
\end{equation*}
and therefore the integration error can be estimated by
\begin{equation*}
\begin{split}
 \EE \left( I(f) - Q^{\CD}(f) \right)^2 &= (\bias(Q^{\CD},f))^2 + \Var (Q^{\CD}(f))\\
&\le {\rm b}_{\mathcal{Q},\bsa}^2 + \sum_{\emptyset \neq u\in\mathcal{Q}} \sum_{w\subseteq u} \Var (Q_{u,n_u}(\Psi_{w,\bsa}(f^+_u))).
\end{split}
\end{equation*}
\end{proof}

For the rest of the paper we assume that the following assumption hold.

\begin{assumptions}\label{(A7)}
Let $c,C,\tau >0$, $\alpha_1 \ge 0$, and $\alpha_2 \in [0,1]$.
Assume that we have for every $\emptyset \neq u\in\mathcal{Q}$ and every $n_u\in\N$
unbiased algorithms $Q_{u,n_u}$ of the form
(\ref{alg-form}) that satisfy (\ref{summe=1}) and condition (*) of Lemma \ref{ANOVA},
and additionally for each $w\subseteq u$ with $\gamma_w >0$
\begin{equation}
 \label{annahme_pw11}
\Var(Q_{u,n_u}(f_w)) \le c C^{|u|}(n_u + 1)^{-\tau}  F_w(n_u) \|f_w\|^2_{k_w}
\hspace{2ex}\text{for all $f_w\in H_w$,}
\end{equation}
where $F_{w}(n) = 1$ for $|w| \le 1$ and
\begin{equation*}
 F_w(n) = \left( 1 + \frac{\ln(n+1)}{(|w|-1)^{\alpha_2}} \right)^{\alpha_1(|w|-1)^{\alpha_2}}
\hspace{2ex}\text{for $|w|\ge 2$.}
\end{equation*}
\end{assumptions}

\begin{remark}\label{Weak}
To achieve our main result Theorem \ref{Theo_UB_General} we may relax
the condition in Assumption \ref{(A7)} that for each $\emptyset \neq u\in\mathcal{Q}$ algorithms
$Q_{u,n_u}$ have to exist for \emph{every $n_u\in\N$}.
It is easily seen that it is sufficient to have those algorithms, e.g., only for all
$b^{m}$, $m = 1,2,\ldots$, where $b\in\{2,3,\ldots\}$ is some suitable fixed base (as it is usually the
case if one employs quasi-Monte Carlo algorithms based on special net constructions).
In Section~\ref{SPLR}
we will rely on this simple observation.
\end{remark}

From (\ref{annahme_pw11}) we obtain with (\ref{variance}) and Lemma \ref{Analogon-Lemma7} for all $f\in B(K)$
\begin{equation}\label{Var_Abs}
 \begin{split}
  \Var(Q^{\CD}(f))
 &\le c \sum_{\emptyset \neq u\in\mathcal{Q}} C^{|u|} (n_u+1)^{-\tau} \sum_{w\subseteq u}  F_w(n_u) \|\Psi_{w,\bsa}(f^+_u)\|_{k_w}^2\\
 &\le c \sum_{\emptyset \neq u\in\mathcal{Q}} C^{|u|} (n_u+1)^{-\tau} \left( \sum_{w\subseteq u} k_{u\setminus w}(\bsa,\bsa) F_w(n_u) \right)
      r^2_{u,u,\bsa} \|f^+_{u}\|^2_K\\
 &\le \sum_{\emptyset \neq u\in\mathcal{Q}}  \mathcal{C}_{u} r^2_{u,u,\bsa} (n_u+1)^{-\tau},
\end{split}
\end{equation}
where
\begin{equation*}
\mathcal{C}_{u} := c C^{|u|} \sum_{w\subseteq u} k_{u\setminus w}(\bsa,\bsa) F_w(n_u).
\end{equation*}
This, estimate (\ref{Int_Err}), and identity (\ref{bias_cda}) lead to
\begin{equation} \label{Fehlerabschaetzung}
e(Q^{\CD},B(K))^2 \le \left( \sum_{\emptyset \neq u\in\mathcal{Q}}  \mathcal{C}_{u} r^2_{u,u,\bsa} (n_u + 1)^{-\tau}
+ \sum_{\emptyset \neq u \in \U} S_{\mathcal{Q},u}^2 \gamma_u k_u(\bsa, \bsa) \right).
\end{equation}
The aim is now to minimize the right hand side of this error bound for given
cost by choosing the set $\mathcal{Q}$ and the quadratures $Q_{u,n_u}$ (essentially) optimal.

\begin{remark}\label{Abgeschlossen}
 The idea of using a changing dimension algorithm in the ANOVA setting is to
approximate the important ANOVA components of the integrand (i.e., the components corresponding
to coordinate sets $u$ with large weights) very well by addressing these components with the help
of an anchored decomposition with a well-chosen anchor $\bsa$. To achieve this it is necessary
that for the set $\mathcal{W}$ of important coordinate sets we have
$S_{\mathcal{Q},u} = 0$ for all $u\in\mathcal{W}$, otherwise the worst case bias $b^2_{\mathcal{Q},\bsa}$
of the changing dimension algorithm $Q^{\CD}$ cannot become small, see (\ref{bias_cda}). (Recall that $S_{\mathcal{Q},u}$ is
an integer, so it is only ``small'' if it vanishes.)
A sufficient condition to achieve this is $\overline{\mathcal{W}} \subseteq \mathcal{Q}$.
Thus it seems a to be a reasonable default choice to take $\mathcal{Q} = \overline{\mathcal{W}}$.
\end{remark}

Let us assume that $\decay_{\bsgamma} >1$. Furthermore, let us choose an anchor of the form
\begin{equation*}
 {\bf a}=(a,a,\dots),
\hspace{2ex}\text{where $a \in D$ satisfies \eqref{anker_a}.}
\end{equation*}
Due to (\ref{Fehlerabschaetzung}) it is advantageous to choose $a\in D$
such that $k(a,a)$ is minimized, if possible, or is at least relatively small.

To define changing dimension algorithms for general weights in the ANOVA setting, we adapt the approach used by
Plaskota and Wasilkowski in \cite{PW11} for product weights in the anchored setting.
Our modifications of the approach in \cite{PW11} allow us to make use of the error estimate
(\ref{Fehlerabschaetzung}) which is based on the ANOVA invariance lemma,
Lemma \ref{ANOVA}, and on the norm estimate from Lemma \ref{Analogon-Lemma7}.


Without loss of generality we may assume that $\tau < \decay_{\bsgamma} -1$; if this is not satisfied,
we simply replace $\tau$ by $\decay_{\bsgamma} -1-\delta$ for some small $\delta>0$. Thus we may
choose an $\alpha_0$ that satisfies
\begin{equation*}
 \frac{\tau}{\decay_{\bsgamma}} < \alpha_0 < 1 -\frac{1}{\decay_{\bsgamma}}.
\end{equation*}
Put
\begin{equation*}
 L_{1-\alpha_0} := \sum_{\emptyset \neq u \in \mathcal{U}} \gamma_u^{1-\alpha_0}
\hspace{2ex}\text{and}\hspace{2ex}
\widehat{C} := \max\{C(1+k(a,a)), 4k(a,a)\},
\end{equation*}
where $C$ is the constant appearing in (\ref{annahme_pw11}).

For a given $\varepsilon >0$ we now choose for each $u\in\U$ a number $n_u'$ as follows:
\begin{equation*}
 n_u' = n_u'(\varepsilon, \alpha_0) :=
\begin{cases}
\, 0
\hspace{2ex}&\text{if $c L_{1-\alpha_0} \widehat{C}^{|u|} \gamma_u^{\alpha_0} < \varepsilon^2$},\\
\, \lfloor (c L_{1-\alpha_0} \widehat{C}^{|u|} \gamma_u^{\alpha_0} \varepsilon^{-2})^{1/\tau} \rfloor
\hspace{2ex} &\text{otherwise.}
\end{cases}
\end{equation*}
The actual number of sample points $n_u$ used by our changing dimension algorithm is then given by
\begin{equation*}
 n_u = n_u(\varepsilon, \alpha_0) :=
\begin{cases}
\, \max \{1,n_u'\}
\hspace{2ex}&\text{if there exists a $v\in\U$ with $u\subseteq v$ and $n_v'\ge 1$},\\
\, 0
\hspace{2ex} &\text{otherwise.}
\end{cases}
\end{equation*}
We put $\mathcal{Q} := \{ u\in \U \,|\, n_u >0 \}$. Notice that our choice of the numbers $n_u$, $u\in\U$, leads
to $\mathcal{Q} = \overline{\mathcal{Q}}$.
Therefore we obtain (cf. Remark \ref{Abgeschlossen})
\begin{equation*}
 S_{\mathcal{Q},u}^2
\begin{cases}
\, = 0
\hspace{2ex}&\text{if $\emptyset \neq u \in \mathcal{Q}$},\\
\, \le 2^{2|u|}
\hspace{2ex} &\text{otherwise,}
\end{cases}
\end{equation*}
which in turn implies
\begin{equation*}
 b_{\mathcal{Q},\bsa}^2 \le \sum_{u \notin \mathcal{Q}} 2^{2|u|} \gamma_u k_u(\bsa, \bsa).
\end{equation*}
If we define
\begin{equation}\label{def_b_eps}
 \mathcal{B}(\varepsilon) := \max_{u \in \mathcal{Q}} \max_{w\subseteq u} F_w(n_u),
\end{equation}
then (\ref{Int_Err}) and (\ref{Var_Abs}) result in
\begin{equation}\label{zellini}
 e(Q^{\CD}, B(K))^2
\le c \sum_{\emptyset \neq u\in\mathcal{Q}} \widehat{C}^{|u|} r^2_{u,u,\bsa} (n_u + 1)^{-\tau} \mathcal{B}(\varepsilon)
+ \sum_{u\notin \mathcal{Q}} \widehat{C}^{|u|} \gamma_u.
\end{equation}
To make this error estimate more explicit, we need to know more about the quantities $r^2_{u,u,\bsa}$,
$\emptyset \notin u \in \mathcal{Q}$, and about $\mathcal{B}(\varepsilon)$.
To study the last quantity more closely, it is helpful to introduce the
\emph{$\varepsilon$-dimension} $d(\varepsilon)$, which is defined to be the \emph{size of the largest group
of active variables} that is used by the changing dimension algorithm $Q^{\CD} = Q^{\CD}_{\varepsilon}$, i.e.,
\begin{equation*}
 d(\varepsilon) := \max\{ |u| \, |\, u\in \mathcal{Q} \}.
\end{equation*}
The $\varepsilon$-dimension is also helpful for the cost analysis, since
\begin{equation}\label{cost_ana}
\begin{split}
 \cost_{}(Q^{\CD}, B(K)) \le &\sum_{u\in\mathcal{Q}} 2^{|u|} \$(|u|) n_u
\le \$(0) + \$(d(\varepsilon)) \sum_{\ell =1}^{d(\varepsilon)} 2^{\ell} \sum_{|u| = \ell} n_u\\
\le &\$(0) + \$(d(\varepsilon)) \sum_{\ell =1}^{d(\varepsilon)} 2^{2\ell} \sum_{|u| = \ell} n_{u'},
\end{split}
\end{equation}
and
\begin{equation*}
 \sum_{|u| = \ell} n_{u'} \le ( L_{1-\alpha_0} c \widehat{C}^{\ell} \varepsilon^{-2} )^{1/\tau} \frac{1}{\ell !} \sum_{|u|=\ell} \gamma_u^{\alpha_0/\tau}.
\end{equation*}
Let us now assume that the following three estimates hold:
\begin{equation}\label{mumford}
 r^2_{u,u,\bsa} = O(\gamma_u) \,,\, \hspace{2ex} d(\varepsilon) = o(\ln(1/\varepsilon))\,,\,
\hspace{2ex}\text{and}\hspace{2ex} \mathcal{B}(\varepsilon) = \varepsilon^{-o(1)}.
\end{equation}
Then we get from (\ref{zellini}) and the definition of the $n_u$s
\begin{equation}
\label{error_pw}
\begin{split}
e(Q^{\CD}, B(K))^2 &= O \left( c \sum_{\emptyset \neq u\in\U} \widehat{C}^{|u|} (n_u +1)^{-\tau} \gamma_u  \mathcal{B}(\varepsilon) \right)\\
&= O \left( \sum_{\emptyset \neq u\in\U} \gamma_u^{1-\alpha_0} \left( c \widehat{C}^{|u|} (n_u+1)^{-\tau}
\gamma^{\alpha_0}_u \right)  \varepsilon^{-o(1)} \right) \\
&= O(\varepsilon^{2-o(1)}).
\end{split}
\end{equation}
Furthermore, we get from (\ref{cost_ana})
\begin{equation*}
\begin{split}
\cost_{}(Q^{\CD}, B(K))
&\le \$(0) + \$(d(\varepsilon)) \sum^{d(\varepsilon)}_{\ell =1} 2^{2\ell}
\left( L_{1-\alpha_0} c \widehat{C}^{\ell} \varepsilon^{-2} \right)^{1/\tau} \frac{1}{\ell !} \sum_{|u|=\ell} \gamma_u^{\alpha_0/\tau}\\
&\le \$(0) + \frac{\$(d(\varepsilon))}{\varepsilon^{2/\tau}}
\left( c L_{1-\alpha_0} \right)^{1/\tau} L_{\alpha_0/\tau} \exp( 4 \widehat{C}^{1/\tau}),
\end{split}
\end{equation*}
resulting in
\begin{equation*}
\cost_{}(Q^{\CD}, B(K)) \le \$(0) + O \left( \$(d(\varepsilon)) \varepsilon^{-2/\min\{\tau, \decay_{\bsgamma} -1 - \delta\}} \right).
\end{equation*}
Thus we have proved the following theorem for general weights.

\begin{theorem}\label{Theo_UB_General}
Let $\$(\nu) = O(e^{\sigma\nu})$ for some $\sigma\in (0,\infty)$.
Let $\bsgamma$ be weights with $\decay_{\bsgamma} >1$. If Assumption \ref{(A7)}
is satisfied, and if in addition (\ref{mumford}) holds, then we have for all $\delta >0$ that
\begin{equation*}
 e(N, B(K))^2 = O \left( N^{-\min\{\tau, \decay_{\bsgamma}-1 \} + \delta} \right),
\end{equation*}
or, equivalently,
\begin{equation*}
p^{{\rm res}}_{} \le \max \left\{ \frac{2}{\tau}\,,\,
\frac{2}{{\rm decay}_{{\bsgamma}} - 1} \right\}.
\end{equation*}
\end{theorem}

\subsection{Finite-intersection weights}\label{UB_FOW}

In this subsection, we consider finite-order weights $\bsgamma$ of order $\beta$. Again we choose an anchor
of the form ${\bf a}=(a,a,\dots)$, where $a \in D$ satisfies \eqref{anker_a}.

For general finite-order weights of order $\beta$ we clearly have
$d(\varepsilon) \le \beta$ and
\begin{equation*}
  \mathcal{B} (\varepsilon) \le \max_{u\in \mathcal{Q}} \max \left\{ 1, (1+\ln(n_u+1))^{\alpha_1(\beta -1)^{\alpha_2}} \right\}
= O(1+\ln(1+\varepsilon^{-1})) = \varepsilon^{-o(1)}.
\end{equation*}
Let us now assume a stronger monotonicity condition than (A6), namely
\begin{equation}\label{monoton}
 \gamma_u \ge \gamma_v
\hspace{2ex}\text{for all $u,v\in \U$ with $u\subseteq v$.}
\end{equation}
Condition (\ref{monoton}) leads for $\emptyset \neq u \in \U$ to
\begin{equation*}
 \begin{split}
  r^2_{u,u,\bsa} = &\gamma_u \sum_{u'\subset \N\setminus u} \frac{\gamma_{u\cup u'}}{\gamma_u} k_{u'}(\bsa,\bsa)
\le \gamma_u \sum_{u'\subset \N\setminus u\,;\, \gamma_{u\cup u'} >0}  k(a,a)^{|u'|}\\
\le &\gamma_u(1+ \varrho) \max\{1, k(a,a)^{\beta}\} = O(\gamma_u).
 \end{split}
\end{equation*}
Thus (\ref{mumford}) holds. Moreover, (\ref{monoton}) implies
\begin{equation*}
 \sum_{j\in\N} \gamma^{(1)}_{\{j\}} \le \sum_{u\in\U} \gamma_u \le \eta \sum_{j\in\N} \gamma^{(1)}_{\{j\}},
\end{equation*}
where $\eta$ is as in (\ref{cond}). In particular, we have $\decay_{\bsgamma} = \decay_{\bsgamma^{(1)}}$.
These observations, Theorem \ref{Theorem3.1''}, Theorem \ref{Theo_UB_General}, and \cite[Thm.~4.3]{BG12} lead to the following result.

\begin{theorem}\label{Theo_UB_FIW}
Let $\$(\nu) = O(e^{\sigma\nu})$ for some $\sigma\in (0,\infty)$.
Let $\bsgamma$ be finite-intersection weights with
${\rm decay}_{{\bf \gamma}} > 1$.
If Assumption \ref{(A7)} and the monotonicity condition (\ref{monoton}) are satisfied
then we have for all $\delta >0$ that
\begin{equation*}
 e(N, B(K))^2 = O \left( N^{-\min\{\tau, \decay_{\bsgamma}-1 \} + \delta} \right),
\end{equation*}
or, equivalently,
\begin{equation*}
p^{{\rm res}}_{} \le \max \left\{ \frac{2}{\tau}\,,\,
\frac{2}{{\rm decay}_{{\bsgamma}} - 1} \right\}.
\end{equation*}
Assume additionally that condition
(\ref{alpha}) is satisfied for $\alpha = \tau$
and that $\$(\nu) = \Omega(\nu)$. Then
\begin{equation*}
p^{{\rm res}}_{} = \max \left\{ \frac{2}{\alpha}\,,\,
\frac{2}{{\rm decay}_{{\bsgamma}} - 1} \right\}.
\end{equation*}
\end{theorem}

\begin{remark}
With the help of suitable \emph{randomized multilevel algorithms}, we may also get sharp upper bounds for the
strong exponent of tractability in the case where the function $\nu \mapsto \$(\nu)$ grows slower than linearly
in $\nu$.  More precisely, we have the following result:

Let $\bsgamma$ be finite-intersection weights with
${\rm decay}_{{\bf \gamma}} > 1$.
Let Assumption \ref{(A7)} and the monotonicity condition (\ref{monoton}) be satisfied.
Assume that additionally condition (\ref{alpha}) is satisfied for $\alpha = \tau$.
Let $\$(\nu) = \Theta(\nu^s)$ for some $s\in (0,\infty)$.
Then we have
\begin{equation*}
p^{{\rm res}}_{} = \max \left\{ \frac{2}{\alpha}\,,\,
\frac{2\min\{1,s\}}{{\rm decay}_{{\bsgamma}} - 1} \right\}.
\end{equation*}
The upper bound on $p^{{\rm res}}_{}$ in the case where $0< s \le 1$ was
derived in \cite[Thm.~4.3]{BG12} with the help of randomized multilevel algorithms.
\end{remark}

\subsection{Product Weights}\label{UB_PW}

In this subsection, we discuss product weights. For product weights assumption (A6) is obviously satisfied and
additionally, due to $\bsa\in\X$, we have always
\begin{equation*}
 r^2_{u,u,\bsa} = \gamma_u \sum_{u'\subset\N \setminus u} \gamma_{u'} k_{u'}(\bsa,\bsa) = O(\gamma_u).
\end{equation*}
Furthermore, due to the definition of product weights and of $Q^{\CD}$, we have
\begin{equation*}
 d(\varepsilon) =\max \left\{ \ell \, \Big|\, c \widehat{C}^{\ell} L_{1-\alpha_0} \gamma^{\alpha_0}_{[\ell]} > \varepsilon^2 \right\},
\end{equation*}
and it was proved in \cite[Lemma~1]{PW11} that this quantity is indeed rather small in terms of $1/\varepsilon$,
namely
\begin{equation}\label{deps_est}
 d(\varepsilon) = O\left( \frac{\ln(1/\varepsilon)}{\ln\ln(1/\varepsilon)} \right) = o \left( \ln(1/\varepsilon) \right).
\end{equation}
Although the quantity $\mathcal{B}(\varepsilon)$ defined in (\ref{def_b_eps}) differs slightly from the quantity $B(\varepsilon)$ defined in
\cite[Sect.~3]{PW11}, we can use exactly the same argument used there for $B(\varepsilon)$ to show that also
$\mathcal{B}(\varepsilon) = \varepsilon^{-o(1)}$ as $\varepsilon \to 0$. We briefly repeat the argument for the convenience of the
reader: For $u\in \mathcal{Q}$ and $w\subseteq u$ with $|u| \ge |w| \ge 2$ we may write
$F_w(n_u) = \varepsilon^{-{\rm power}(w,n_u)}$ with
\begin{equation*}
 {\rm power}(w,n_u) := \alpha_1 \frac{(|w|-1)^{\alpha_2}}{\ln(1/\varepsilon)} \ln \left( 1+ \frac{\ln(1+n_u)}{(|w|-1)^{\alpha_2}} \right).
\end{equation*}
Put
$x:= (|w|-1)^{\alpha_2}/\ln(1/\varepsilon)$.
Then we have ${\rm power}(w,n_u) \le \alpha_1 x \ln(1+ \overline{c}/x)$ for some suitable constant $\overline{c}$.
Since $|w| \le |u| \le d(\varepsilon) = o(\ln(q/\varepsilon))$, the maximal value of $x$ tends to zero as $\varepsilon$
approaches zero, and the same holds for $x\ln(1+\overline{c}/x)$. Hence
\begin{equation*}
 \max_{u\in \mathcal{Q}} \max_{w\subseteq u\,;\, |w| \ge 2} {\rm power}(w,n_u) = o(1).
\end{equation*}

Therefore we obtain with Theorem \ref{Theo_UB_General}, Theorem \ref{Theorem3.1''}, and \cite[Thm.~4.5]{BG12}
the following result.

\begin{theorem}\label{Theo_UB_PW}
Let $\$(\nu) = O(e^{\sigma \nu})$ for some $\sigma\in (0,\infty)$.
Let $\bsgamma$ be product weights, and let
${\rm decay}_{{\bf \gamma}} > 1$.
If Assumption \ref{(A7)} is satisfied,
then we have for all $\delta >0$ that
\begin{equation*}
 e(N, B(K))^2 = O \left( N^{-\min\{\tau, \decay_{\bsgamma}-1 \} + \delta} \right),
\end{equation*}
or, equivalently,
\begin{equation*}
p^{{\rm res}}_{} \le \max \left\{ \frac{2}{\tau}\,,\,
\frac{2}{{\rm decay}_{{\bsgamma}} - 1} \right\}.
\end{equation*}
Assume additionally that condition
(\ref{alpha}) is satisfied for $\alpha = \tau$ and
that $\$(\nu) = \Omega(\nu)$. Then
\begin{equation*}
p^{{\rm res}}_{} = \max \left\{ \frac{2}{\alpha}\,,\,
\frac{2}{{\rm decay}_{{\bsgamma}} - 1} \right\}.
\end{equation*}
\end{theorem}

For the lower bound on $p^{{\rm res}}_{}$ notice that for product weights we always
have $\decay_{\bsgamma} = \decay_{\bsgamma^{(1)}}$, see, e.g., \cite[Thm.~5]{DG12}.

\begin{remark}
Similarly as in the case of finite-intersection weights, we may obtain better upper bounds for the
strong exponent of tractability in the case where the function $\nu \mapsto \$(\nu)$ grows slower than linearly
in $\nu$ by using suitable randomized multilevel algorithms instead of changing dimension algorithms.
More precisely, the following was shown in \cite[Thm.~4.5]{BG12}:
Let $\$(\nu) = \Theta(\nu^s)$ for some $s$ satisfying $\frac{\alpha -1}{\alpha} \le s \le 1$.
Under condition (\ref{alpha}) and an additional assumption that slightly differs from  Assumption \ref{(A7)}
we have
\begin{equation*}
p^{{\rm res}}_{} = \max \left\{ \frac{2}{\alpha}\,,\,
\frac{2\min\{1,s\}}{{\rm decay}_{{\bsgamma}} - 1} \right\}.
\end{equation*}
In the case where $0<s<(\alpha-1)/\alpha$ we still have good upper and lower bounds for $p^{{\rm res}}_{}$,
but unfortunately they do not match anymore; for details see \cite[Thm.~4.5]{BG12}.
\end{remark}

\section{Examples: Unanchored Sobolev Spaces and Scrambled Polynomial Lattice Rules}
\label{SPLR}

In this section we specialize to a concrete example of a reproducing kernel Hilbert space of smoothness $\chi \in \mathbb{N}$ and explicit constructions of quadrature rules which satisfy Assumption~\ref{(A7)}.

\subsection{Unanchored Sobolev Spaces} \label{subsecunachoredSobspace}

We consider now the domain $D = [0,1]$ where $\Sigma$ is the Borel $\sigma$ algebra and $\rho$ the Lebesgue measure. The following reproducing kernel Hilbert space was, for instance, considered in \cite{BD09, KSWW10b}. For arbitrary $\chi \in \mathbb{N}$
we study numerical integration in the reproducing kernel Hilbert space $H(K_{\chi})$ with reproducing kernel
\begin{equation*}
K_{\chi}(\bsx, \bsy) = \sum_{u \in \mathcal{U} } \gamma_u \prod_{j\in u} k_{\chi}(x_j, y_j),
\end{equation*}
where
\begin{equation*}
k_\chi(x_j,y_j) = \sum_{\tau=1}^\chi \frac{B_\tau(x_j)}{\tau!} \frac{B_\tau(y_j)}{\tau!} + (-1)^{\chi+1} \frac{B_{2\chi}(|x_j-y_j|)}{(2\chi)!},
\end{equation*}
and where $B_\tau$ is the Bernoulli polynomial of degree $\tau$. Let $k_{\chi,u}(\bsx, \bsy) = \prod_{j\in u} k_\chi(x_j, y_j)$.
Note that $k_\chi$ satisfies Assumptions (A~1), (A~2), (A~3), (A~5), and (A 2a).

In one dimension, the inner product in $H(k_{\chi})$ is given by
\begin{equation*}
\langle f, g \rangle_{k_{\chi}} = \sum_{\tau=1}^{\chi-1} \int_0^1 f^{(\tau)}(x) \mathrm{d} x \int_0^1 g^{(\tau)}(x) \mathrm{d} x + \int_0^1 f^{(\chi)}(x) g^{(\chi)}(x) \mathrm{d} x
\end{equation*}
and the norm is given by
\begin{equation*}
\|f\|_{k_\chi} = \left(\sum_{\tau=1}^{\chi-1} \left(\int_0^1 f^{(\tau)}(x) \mathrm{d} x \right)^2 + \int_0^1 |f^{(\chi)}(x)|^2 \mathrm{d} x \right)^{1/2};
\end{equation*}
here $f^{(\tau)}$ and $g^{(\tau)}$, $\tau =1,\ldots,\chi$,  are the $\tau$th-distributional derivatives of
$f$ and $g$, respectively.
The norm in $H(k_{\chi, u})$ is given by
\begin{equation*}
\|f\|_{k_{\chi, u}} = \left(\sum_{v \subseteq u} \sum_{\bstau \in \{\chi\}^{|v|} \times [\chi-1]^{|u|-|v|} } \int_{[0,1]^{|v|}}
\left| \int_{[0,1]^{|u|-|v|}} \frac{\partial^{|\bstau|} f}{\prod_{j \in u} \partial x_j^{\tau_j} } \rd \bsx_{u\setminus v} \right|^2
\rd \bsx_v \right)^{1/2},
\end{equation*}
where for a multi-index $\bstau$ we denote the sum $\sum_{j\in u} \tau_j$ by $|\bstau|$.

For all $v\in\U$ let $K_{\chi,v} := \sum_{u \subseteq v} \gamma_u k_{\chi, u}$. For $f \in H(K_{\chi, v})$ we have the unique decomposition
\begin{equation}\label{eq_anova}
f = \sum_{u \subseteq v} f_u,
\end{equation}
where $f_u \in H(k_{\chi, u})$. Note that \eqref{eq_anova} is the ANOVA decomposition of $f$.

\subsection{Polynomial lattice rules}

We introduce some notation first. For a prime $b$, let $\FF_b$ be the finite field containing $b$ elements $\{0,\ldots, b-1 \}$ and by $\FF_b((x^{-1}))$ we denote the field of formal Laurent series over $\FF_b$. Every element of $\FF_b((x^{-1}))$ is of the form
  \begin{align*}
    L = \sum_{l=w}^{\infty}t_l x^{-l} ,
  \end{align*}
where $w$ is an arbitrary integer and all $t_l\in \FF_b$. Further, we denote by $\FF_b[x]$ the set of all polynomials over $\FF_b$. For a given $m\in\N$,
we define the map $v_m$ from $\FF_b((x^{-1}))$ to the interval $[0,1)$ by
  \begin{align*}
    v_m\left( \sum_{l=w}^{\infty}t_l x^{-l}\right) =\sum_{l=\max(1,w)}^{m}t_l b^{-l}.
  \end{align*}
We often identify $k\in \nat_0$, whose $b$-adic expansion is given by $k=\kappa_0+\kappa_1 b+\cdots +\kappa_{a-1} b^{a-1}$, with the polynomial over $\FF_b[x]$ given by $k(x)=\kappa_0+\kappa_1 x+\cdots +\kappa_{a-1} x^{a-1}$.  For $\bsk=(k_1,\ldots, k_s)\in (\FF_b[x])^s$ and $\bsq=(q_1,\ldots, q_s)\in (\FF_b[x])^s$,
we define the ``inner product'' as
  \begin{align*}
     \bsk \cdot \bsq =\sum_{j=1}^{s}k_j q_j \in \FF_b[x] ,
  \end{align*}
and we write $q\equiv 0 \pmod p$ if $p$ divides $q$ in $\FF_b[x]$.

The definition of a polynomial lattice rule is given as follows.
\begin{definition}\label{def:polynomial_lattice}
Let $b$ be prime and $m, s \in \nat$. Let $p \in \FF_b[x]$ be an irreducible polynomial with $\deg(p)=m$ and let $\bsq=(q_1,\ldots,q_s) \in (\FF_b[x])^s$. Now we construct a point set consisting of $b^m $points in $[0,1)^s$ in the following way:  For $0 \le h < b^m$, identify each $h$ with a polynomial $h(x)\in \FF_b[x]$ of $\deg(h(x))<m$. Then the $h$-th point is obtained by setting
  \begin{align*}
    \bsx_h
    &:=
    \left( v_m\left( \frac{h(x) \, q_1(x)}{p(x)} \right) , \ldots , v_m\left( \frac{h(x) \, q_s(x)}{p(x)} \right) \right) \in [0,1)^s .
  \end{align*}
The point set $\{\bsx_0, \bsx_1,\ldots, \bsx_{b^m-1}\}$ is called a \emph{polynomial lattice point set} and a QMC rule using this point set is called a \emph{polynomial lattice rule} with generating vector $\bsq$ and modulus $p$.
\end{definition}

\subsection{Owen's scrambling}

We now introduce Owen's scrambling algorithm. This procedure is best explained by using only one point $\bsx$. We denote the point obtained after scrambling $\bsx$ by $\bsy$. For $\bsx=(x_1,\ldots, x_s)\in [0,1)^s$, we denote the $b$-adic expansion by
  \begin{align*}
     x_j=\frac{x_{j,1}}{b}+\frac{x_{j,2}}{b^2}+\cdots ,
  \end{align*}
for $1\le j\le s$. Let $\bsy=(y_1,\ldots, y_s)\in [0,1)^s$ be the scrambled point whose $b$-adic expansion is represented by
  \begin{align*}
     y_j=\frac{y_{j,1}}{b}+\frac{y_{j,2}}{b^2}+\cdots ,
  \end{align*}
for $1\le j\le s$. Each coordinate $y_j$ is obtained by applying random permutations to each digit of $x_j$. Here the permutation applied to
$x_{j,k}$ depends on $x_{j,l}$ for $1\le l\le k-1$. In particular, $y_{j,1}=\pi_j(x_{j,1})$, $y_{j,2}=\pi_{j,x_{j,1}}(x_{j,2})$, $y_{j,3}=\pi_{j,x_{j,1},x_{j,2}}(x_{j,3})$,
and in general
  \begin{align*}
     y_{j,k}=\pi_{j,x_{j,1},\ldots, x_{j,k-1}}(x_{j,k}) ,
  \end{align*}
where $\pi_{j,x_{j,1},\ldots, x_{j,k-1}}$ is a random permutation of $\{0,\ldots, b-1\}$. We choose permutations with different indices mutually independent from each other
where each permutation is chosen uniformly distributed. Then, as shown in \cite[Proposition~2]{O95}, the scrambled point $\bsy$ is uniformly distributed in $[0,1)^s$.

In order to simplify the notation, we denote by $\Pi_j$ the set of permutations associated with the $j$th variable, that is,
  \begin{align*}
     \Pi_j=\{ \pi_{j,x_{j,1},\ldots, x_{j,k-1}}: k\in \nat, x_{j,1},\ldots, x_{j,k-1}\in \{0,\ldots, b-1 \}\} ,
  \end{align*}
and let $\boldsymbol{\Pi}=(\Pi_1,\ldots, \Pi_s)$. We simply write $\bsy=\boldsymbol{\Pi}(\bsx)$ when $\bsy$ is obtained by applying Owen's scrambling to $\bsx$ using the permutations in $\boldsymbol{\Pi}$.

\subsection{Interlaced scrambled polynomial lattice rules}

For the results below we use interlaced scrambled polynomial lattice rules, which we define in the following. We first define the interlacing function.

\begin{definition}\rm
For an integer $\alpha \ge 1$ the digit interlacing function (with interlacing factor $\alpha$) is defined by
\begin{eqnarray*}
\mathscr{D}_\alpha: [0,1)^{\alpha} & \to & [0,1) \\
(x_1,\ldots, x_{\alpha}) &\mapsto & \sum_{d=1}^\infty \sum_{r=1}^\alpha
\xi_{r,d} b^{-r - (d-1) \alpha},
\end{eqnarray*}
where $x_r = \xi_{r,1} b^{-1} + \xi_{r,2} b^{-2} + \cdots$ for $1
\le r \le \alpha$. We also define this function for vectors by setting
\begin{eqnarray*}
\mathscr{D}_\alpha: [0,1)^{\mathbb{N}} & \to & [0,1)^\mathbb{N} \\
(x_1, x_2, \ldots) &\mapsto & (\mathscr{D}_\alpha(x_1,\ldots, x_\alpha),  \mathscr{D}_\alpha(x_{\alpha +1},\ldots, x_{2\alpha}), \ldots),
\end{eqnarray*}
and for point sets $\{\bsx_0,\bsx_1, \ldots, \bsx_{N-1}\} \subseteq [0,1)^{\mathbb{N} }$ by
\begin{equation*}
\mathscr{D}_\alpha(\{\bsx_0,\ldots, \bsx_{N-1}\}) = \{\mathscr{D}_\alpha(\bsx_0), \mathscr{D}_\alpha(\bsx_1), \ldots, \mathscr{D}_\alpha(\bsx_{N-1})\}\subseteq[0,1)^{\mathbb{N}}.
\end{equation*}
\end{definition}

We can now define interlaced scrambled polynomial lattice rules.
\begin{definition}
Let the point set $\{\bsx_0, \bsx_1,\ldots, \bsx_{b^m-1}\}$ be a polynomial lattice point set and let $\boldsymbol{\Pi}$ be a randomly chosen set of permutations. Then the point set
\begin{equation*}
\mathscr{D}_{\alpha}(\{\boldsymbol{\Pi}(\bsx_0), \ldots, \boldsymbol{\Pi}(\bsx_{b^m-1})\})
\end{equation*}
is an interlaced scrambled polynomial lattice point set. A QMC rule using an interlaced scrambled polynomial lattice point set is called an interlaced scrambled polynomial lattice rule.
\end{definition}

\subsection{Results}

The following theorem follows by substituting Lemma~\ref{lem_sigma_bound} from the appendix in the beginning of the proof of \cite[Corollary~1]{GD12} and 
using \cite[Theorem~1]{GD12} (where we choose $\alpha = d = d_0 = \chi$ and $r_0 = \chi s$).

\begin{theorem}\label{theorem:cbc_proof}
Let $b$ be a prime and $m\in \N$. Then an interlaced scrambled polynomial lattice rule $Q_{u,n_u}$, using $n_u = b^m$ points,
can be constructed component-by-component such that for any $f \in H(K_{\chi, u})$ we have
\begin{align*}
\mathrm{Var}(Q_{u,n_u}(f)) \le & (n_u - 1)^{-\frac{1}{\lambda}} \left[ \sum_{\emptyset \neq v \subseteq u}\gamma_v^\lambda D_{\chi,\lambda}^{|v|+1}
\right]^{\frac{1}{\lambda}} \|f\|^2_{K_{\chi, u} },
\end{align*}
for all $1/(2 \chi+1)<\lambda\le 1$, where
\begin{align*}
D_{\chi,\lambda}= -1+(1+C_{b,\chi,\lambda})^\chi,
\end{align*}
and
\begin{align*}
C_{b,\chi,\lambda}=\max\left\{\left(\frac{4^{\chi} (b-1)}{1-b^{- 2\chi}}\right)^{\lambda}, \frac{4^{\lambda \chi} (b-1)}{1-b^{1-(2 \chi+1)\lambda}}\right\} .
\end{align*}	
\end{theorem}

If we choose the weights $\gamma_u$, $u\in\U$, in Theorem \ref{theorem:cbc_proof} to be all equal to one, we obtain the following corollary.

\begin{corollary}\label{corollary:cbc_proof_b}
Let $b$ be a prime and $m\in\N$. Then an interlaced scrambled polynomial lattice rule $Q_{u,n_u}$, using $n_u = b^m$ points in $[0,1]^{|u|}$, can be constructed component-by-component such that for any $f_w \in H(k_{\chi, w})$ with $w \subseteq u$ we have
\begin{align*}
\mathrm{Var}(Q_{u,n_u}(f_w)) \le & (n_u - 1)^{-\tau} D^{\tau}_{\chi, 1/\tau} (1 + D_{\chi,1/\tau})^{|u| \tau}  \|f_w\|^2_{k_{\chi, w}},
\end{align*}
for all $1 \le \tau < 2 \chi+1$, where
$D_{\chi, 1/\tau}$ is defined as in Theorem \ref{theorem:cbc_proof}.
\end{corollary}

Note that Corollary~\ref{corollary:cbc_proof_b} ensures that Assumption~\ref{(A7)} (weakened in the sense of Remark \ref{Weak}) is satisfied
with $\alpha_1 = 0$.
Thus we deduce from Theorem~\ref{Theo_UB_General} the following result for general weights.
\begin{corollary}\label{CorUB_General}
Let $\$(\nu) = O(e^{\sigma\nu})$ for some $\sigma\in (0,\infty)$.
Let $\bsgamma$ be weights with $\decay_{\bsgamma} >1$. If (\ref{mumford}) holds, then we have for all $\delta >0$ that
\begin{equation*}
 e(N, B(K_{\chi}))^2 = O \left( N^{-\min\{2\chi+1, \decay_{\bsgamma}-1 \} + \delta} \right),
\end{equation*}
or, equivalently,
\begin{equation*}
p^{{\rm res}} \le \max \left\{ \frac{2}{2\chi +1}\,,\,
\frac{2}{{\rm decay}_{{\bsgamma}} - 1} \right\}.
\end{equation*}
\end{corollary}

For finite-intersection weights and product weights we can deduce the next result which follows from Theorem~\ref{Theo_UB_FIW} and \ref{Theo_UB_PW}, and Corollary~\ref{corollary:cbc_proof_b}. Note that condition (\ref{alpha}) is satisfied for $\alpha = 2\chi +1$ for the space $H(K_\chi)$ since this already holds for the one-variable case, see \cite[Section~2.2.9, Proposition~1(ii)]{Nov88}.

\begin{corollary}\label{Cor_UB_FIW}
Let $\$(\nu) = O(e^{\sigma\nu})$ for some $\sigma\in (0,\infty)$.
Let $\bsgamma$ be finite-intersection weights that satisfy
the monotonicity condition (\ref{monoton}) or let $\bsgamma$ be product weights.
Assume that ${\rm decay}_{{\bf \gamma}} > 1$. Then we have for all $\delta >0$ that
\begin{equation*}
 e(N, B(K_\chi))^2 = O \left( N^{-\min\{2\chi +1, \decay_{\bsgamma}-1 \} + \delta} \right),
\end{equation*}
or, equivalently,
\begin{equation*}
p^{{\rm res}}_{} \le \max \left\{ \frac{2}{2\chi +1}\,,\,
\frac{2}{{\rm decay}_{{\bsgamma}} - 1} \right\}.
\end{equation*}
Assume additionally that $\$(\nu) = \Omega(\nu)$. Then
\begin{equation*}
p^{{\rm res}}_{} = \max \left\{ \frac{2}{2\chi+1}\,,\,
\frac{2}{{\rm decay}_{{\bsgamma}} - 1} \right\}.
\end{equation*}
\end{corollary}

\section*{Appendix}

In \cite[Section~3.1]{GD12} (see also \cite[Section~3.2]{D11a}) a variation $V_{\chi} = V_{\chi, \bsgamma}$ was defined, which for functions $f$ with continuous partial derivatives of order up to $\chi$ in each variable, is given by
\begin{equation*}
V_{\chi, \bsgamma}(f) = \left( \sum_{u \subseteq [d]} \gamma_u^{-1} \sum_{\boldsymbol{\tau} \in [\chi]^{|u|}} \int_{[0,1]^{|u|}}
\left| \int_{[0,1]^{d-|u|}} \frac{\partial^{|\boldsymbol{\tau}|} f}{\prod_{j\in u} \partial x_j^{\tau_j} } \rd \bsx_{[d]\setminus u} \right|^2 \rd \bsx_u \right)^{1/2},
\end{equation*}
where for $u = \emptyset$ we set $\boldsymbol{\tau} = \bszero$. For instance, for $d=1$ we have
\begin{equation*}
V_{\chi, \bsgamma}(f) = \left(\left( \gamma_{\emptyset}^{-1} \int_0^1 f(x) \rd x\right)^2
+ \gamma_{\{1\}}^{-1} \sum_{\tau=1}^\chi \int_0^1 |f^{(\tau)}(x)|^2 \rd x\right)^{1/2}.
\end{equation*}
By the tensor product structure of the Hilbert spaces $H(k_{\chi,u})$ it follows that
\begin{equation*}
 \sum_{u \subseteq [d]} \gamma_u^{-1} \|f_u\|^2_{k_{\chi, u}} \le V^2_{\chi, \bsgamma}(f).
\end{equation*}

We now show that a reverse estimate also holds. We consider $d=1$ first. Let $0 \le \tau < \chi$. If $f^{(\tau)}$ is absolutely continuous, then the fundamental theorem of calculus gives us
\begin{equation*}
f^{(\tau)}(x) = \int_0^1 f^{(\tau)}(y) \rd y - \int_0^1 \int_x^y f^{(\tau+1)}(z) \rd z \rd y.
\end{equation*}
Using this formula we deduce
\begin{align*}
|f^{(\tau)}(x)|^2 = & \left(\int_0^1 f^{(\tau)}(y) \rd y \right)^2 - 2 \int_0^1 f^{(\tau)}(y) \rd y \int_0^1 \int_x^y f^{(\tau+1)}(z) \rd z \rd y \\ & + \left(\int_0^1 \int_x^y f^{(\tau+1)}(z) \rd z \rd y \right)^2 \\ \le & 2 \left[\left(\int_0^1 f^{(\tau)}(y) \rd y\right)^2 + \left(\int_0^1 |f^{(\tau+1)}(z)| \rd z \right)^2 \right] \\ \le & 2 \left[\left(\int_0^1 f^{(\tau)}(y) \rd y \right)^2 + \int_0^1 |f^{(\tau+1)}(z)|^2 \rd z \right].
\end{align*}
Thus we obtain
\begin{equation}\label{sobolev_estimate}
\int_0^1 |f^{(\tau)}(x)|^2 \rd x \le 2 \left[\left(\int_0^1 f^{(\tau)}(y) \rd y\right)^2 + \int_0^1 |f^{(\tau+1)}(z)|^2 \rd z \right].
\end{equation}
By repeated application of this formula we obtain
\begin{equation*}
\sum_{\tau=1}^\chi \int_0^1 |f^{(\tau)}(x)|^2 \rd x \le (2^{\chi}-1) \left[\sum_{\tau=1}^{\chi-1} \left(\int_0^1 f^{(\tau)}(x) \rd x \right)^2 + \int_0^1 |f^{(\chi)}(x)|^2 \rd x \right].
\end{equation*}

Let $P_u$ denote the set of polynomials defined on $[0,1]^{u}$. Thus, by applying the above formula in each coordinate, we obtain for functions $f \in P_u$
that
\begin{equation}\label{ineq_V_k_u}
V_{\chi,\bsgamma}^2(f) \le \gamma_u^{-1}(2^{\chi}-1)^{|u|} \|f\|^2_{k_{\chi,u}}.
\end{equation}
Thus, for $f \in P_{[d]}$ we have
\begin{equation*}
\sum_{u \subseteq [d]} \gamma_u^{-1} \|f\|^2_{k_{\chi,u}} \le V_{\chi,\bsgamma}^2(f) \le \sum_{u \subseteq [d]} \gamma_u^{-1} (2^{\chi }-1)^{|u|} \|f\|^2_{k_{\chi,u}},
\end{equation*}
and in particular, for $\chi=1$ we have equality.

Let $r \ge 1$ and $k_1,\ldots, k_r \in \mathbb{N}_0$. Let $k_i =
\kappa_{i,0} + \kappa_{i,1} b + \cdots$, where $\kappa_{i,a} \in
\{0,\ldots, b-1\}$ and $\kappa_{i,a} = 0$ for $a$ large enough. We define a digit interlacing function $\mathscr{E}_r$ for natural numbers by
\begin{eqnarray*}
\mathscr{E}_r: \mathbb{N}^{r} & \to & \mathbb{N} \\
(k_1,\ldots, k_{r}) &\mapsto & \sum_{a=0}^\infty \sum_{z=1}^r
\kappa_{z,a} b^{z -1 + a r}.
\end{eqnarray*}
We also extend this function to vectors
\begin{eqnarray*}
\mathscr{E}_r: \mathbb{N}^{dr} & \to & \mathbb{N}^d \\ (k_1,\ldots, k_{dr}) & \mapsto & (\mathscr{E}_r(k_1,\ldots, k_r), \ldots, \mathscr{E}_r(k_{r(d-1)+1}, \ldots, k_{dr})).
\end{eqnarray*}

Let $d \ge 1$ and $\bsell = (\bsell_1,\ldots, \bsell_d) \in \mathbb{N}_0^{d r}$,
where $\bsell_i = (l_{(i-1)r+1},\ldots, l_{ir})$. Let
\begin{equation*}
B_{r,\bsell,d} = \{(k_1,\ldots, k_{d r}) \in \mathbb{N}_0^{d r}:
\lfloor b^{l_i-1} \rfloor \le k_i < b^{l_i} \mbox{ for } 1 \le i \le
d r\}.
\end{equation*}
Let $\widehat{f}(\bsk) = \int_{[0,1]^d} f(\bsx) \overline{\wal_{\bsk}(\bsx)} \rd \bsx$ denote the Walsh coefficient of $f$. For $\bsell_u \in \mathbb{N}^{|u|}$ we set
\begin{equation*}
\sigma^2_{r,(\bsell_u, \bszero) ,d}(f) =  \sum_{\bsk \in B_{r, (\bsell_u, \bszero) ,d}}
\left|\widehat{f}(\mathscr{E}_r(\bsk)) \right|^2.
\end{equation*}
For $u \subseteq \{1, 2, \ldots, dr\}$ we define the set $v(u) \subseteq \{1,\ldots, d\}$ as the set of $1 \le j \le d$
such that $u \cap \{(j-1) r, (j-1) r+1, \ldots, j r \} \neq \emptyset$. It is straightforward to show that
\begin{equation*}
\sigma^2_{r, (\bsell_u, \bszero), d}(f)= \sigma^2_{r, (\bsell_u, \bszero), d}(f_{v(u)}),
\end{equation*}
where $f_{v(u)} \in H(k_{\chi, v(u)})$ is the ANOVA component of $f$ of the set $v(u) \subseteq [d]$.

Let $\mu(0) = 0$ and for $k \in \mathbb{N}$ with $k = \kappa_0 + \kappa_1 b + \cdots \kappa_{a-1} b^{a-1}$, where $\kappa_i \in \{0,1, \ldots, b-1\}$ and $\kappa_{a-1} \neq 0$, we set $\mu(k) = a$. In \cite[Lemma~9]{D11a} and \cite[Section~3.1]{GD12} a bound on $\sigma_{r,\bsell,d}(f)$ was proven of the form
\begin{equation}\label{sigma_bound}
\sigma_{r, (\bsell_u, \bszero), d}(f) \le 2^{|v(u)| \max(r-\chi,0)}  \gamma_{v(u)}^{1/2} \prod_{j\in u} b^{-\min(\chi, r) \mu(k_j)} V_\chi(f),
\end{equation}
where $\bsk_u = (k_j)_{j \in u}$ is such that $(\bsk_u, \bszero) \in B_{r, (\bsell_u, \bszero), d}$. The aim is now to show that the above inequality also holds when one replaces $V_{\chi}(f)$ by $\|f\|_{K_{\chi, [d]}}$ (with a different constant, see below). The proof proceeds by showing the result for a dense subset of $H(K_{\chi,[d]})$ and then extending the result to all functions in $H(K_{\chi,[d]})$. In the following we show that the set $P_{[d]}$ is dense in $H(K_{\chi,[d]})$.
\begin{lemma}
The set of polynomials $P_{[d]}$ is dense in $H(K_{\chi,[d]})$.
\end{lemma}

\begin{proof}
We consider the case $d=1$ first. Let $f \in H(K_{\chi,1})$. By the Stone-Weierstra{\ss} approximation theorem, the fact that continuous functions are dense in $L_2([0,1])$ and the fact that $f^{(\chi)} \in L_2([0,1])$ implies that for any $\varepsilon > 0$ there exists a polynomial $q_0 \in P_{\{1\}}$ such that
\begin{equation*}
\int_0^1 (f^{(\chi)}(x)-q_0(x))^2 \rd x < \varepsilon.
\end{equation*}
Let
\begin{equation*}
q_1(x) = f^{(\chi-1)}(0) + \int_0^x q_0(t) \rd t.
\end{equation*}
Then we have
\begin{align*}
f^{(\chi-1)}(x) - q_1(x) = & \int_0^x (f^{(\chi)}(t) - q_0(t)) \rd t.
\end{align*}
Using this equality we obtain
\begin{align*}
\int_0^1 (f^{(\chi-1)}(x) - q_1(x))^2 \rd x = & \int_0^1 \left[ \int_0^x (f^{(\chi)}(t) - q_0(t)) \rd t \right]^2 \rd x \\ \le & \int_0^1 \int_0^x 1 \rd t \int_0^x (f^{(\chi)}(t)-q_0(t))^2 \rd t \rd x \\ \le & \int_0^1 x \int_0^1 (f^{(\chi)}(t) - q_0(t))^2 \rd t \rd x \\ < & \frac{\varepsilon}{2}.
\end{align*}
By repeating this argument we obtain a sequence of polynomials $q_0, q_1, \ldots, q_\chi$ such that
\begin{equation*}
\left(\int_0^1 (f^{(\tau)}(t) - q_{\chi-\tau}(t) \right)^2 \le \int_0^1 (f^{(\tau)}(t) - q_{\chi-\tau}(t))^2 \rd t < 2^{\tau-\chi} \varepsilon.
\end{equation*}
This shows that $p_\chi$ satisfies $\|f-p_\chi\|_{K_{\chi,1}} < 2 \varepsilon$ and therefore $P_{\{1\}}$ is dense in $H(K_{\chi,1})$.

For arbitrary dimension $d \ge 1$ we have that for any $f \in H(K_{\chi,[d]})$ and $\varepsilon > 0$ there exists a $q_{\bszero} \in P_{[d]}$ such that $\int_{[0,1]^{d}} (f^{(\boldsymbol{\chi})}(\bsx)-q_{\bszero}(\bsx))^2 \rd \bsx < \varepsilon$. The construction for the case $d =1$ can now be applied component-wise to obtain a polynomial $q_{\boldsymbol{\chi}}$ for which $\|f-q_{\boldsymbol{\chi}}\|^2_{K_{\chi,[d]}} < 2^d \varepsilon$. Thus the result follows.
\end{proof}

Let now $f \in H(K_{\chi,[d]})$.   Then there exists a sequence of functions $(f_i)_{i \ge 1}$  in $P_{[d]}$ such that
$\|f-f_i\|_{K_{\chi,[d]}} \to 0$ as $i \to \infty$. This implies that $\|f-f_i\|_{L_2} \to 0$ as $i \to \infty$.

Since for $f_i \in P_u$ we have $V_\chi(f_i) \le 2^{\chi |u|} \gamma_u^{-1/2} \|f_i\|_{k_{\chi, u}}$, we obtain from \eqref{sigma_bound} 
\begin{equation*}
\sigma_{r, (\boldsymbol{l}_u, \bszero), d}(f_i) \le 2^{|v(u)| \max(\chi, r)}  \prod_{j\in u} b^{-\min(\chi, r) \mu(k_j)} \|f_i\|_{k_{\chi,v(u)}}.
\end{equation*}

For any $\bsk \in \mathbb{N}^d$ we have
\begin{equation*}
|\widehat{f}(\bsk) - \widehat{f}_i(\bsk) |  = | \langle f - f_i, \wal_{\bsk} \rangle_{L_2} | \le \|f-f_i\|_{L_2}
\end{equation*}
and therefore
\begin{equation*}
\left| |\widehat{f}(\bsk)|^2 - |\widehat{f}_i(\bsk)|^2 \right| \le \|f-f_i\|_{L_2} \left[\|f\|_{L_2} + \|f_i\|_{L_2} \right] \le \|f-f_i\|_{L_2} [2 \|f\|_{L_2} + \|f-f_i\|_{L_2} ].
\end{equation*}
Let now $\bsell_u \in \mathbb{N}^{u}$ with $\emptyset \neq u \subseteq \{1,2,\ldots, dr\} $, $\varepsilon > 0$ and $i \in \mathbb{N}$ be such that  $$\|f-f_i\|_{L_2} (2 \|f\|_{L_2} + \|f-f_i\|_{L_2}) b^{\|\bsell_u\|_1} \le \varepsilon.$$ Then we have
\begin{align*}
|\sigma_{r, (\bsell_u, \bszero), d}(f) - \sigma_{r, (\bsell_u, \bszero), d}(f_i) | \le & \sum_{\bsk \in B_{r,(\bsell_u, \bszero),d}} \left| \left|\widehat{f}(\mathscr{E}_r(\bsk)) \right|^2 - \left|\widehat{f}_i(\mathscr{E}_r(\bsk)) \right|^2 \right| \\ \le & \|f-f_i\|_{L_2} (2 \|f\|_{L_2} + \|f-f_i\|_{L_2}) b^{\|\bsell_u\|_1} \le \varepsilon.
\end{align*}
Thus we obtain
\begin{align*}
\sigma_{r, (\bsl_u, \bszero), d}(f) \le & \sigma_{r, (\bsl_u, \bszero), d}(f_i) + \varepsilon \\  \le & 2^{|v(u)| \max(\chi, r)}  \prod_{j\in u} b^{-\min(\chi, r) \mu(k_j)} \|f_i\|_{k_{\chi, v(u)}} + \varepsilon.
\end{align*}
Since $\|f_i\|_{K_{\chi, [d]}} \le \|f\|_{K_{\chi, [d]}} + \varepsilon$ we obtain
\begin{align*}
\sigma_{r, (\bsl_u,\bszero) , d}(f) \le & \sigma_{r, (\bsl_u, \bszero), d}(f_i) + \varepsilon \\
\le &  \gamma_{v(u)}^{1/2}  2^{|v(u)| \max(\chi, r)}  \prod_{j\in u} b^{-\min(\chi, r) \mu(k_j)} \|f\|_{K_{\chi,[d]}} \\ & + \left(1 + 2^{|v(u)| \max(\chi, r)}  \prod_{j\in u} b^{-\min(\chi, r) \mu(k_j)} \right) \varepsilon.
\end{align*}
Since $\varepsilon > 0$ can be chosen arbitrarily small, we obtain
\begin{equation*}
\sigma_{r, (\bsl_u, \bszero), d}(f) \le  \gamma_{v(u)}^{1/2} 2^{|v(u)| \max(\chi, r)}  \prod_{j\in u} b^{-\min(\chi, r) \mu(k_j)} \|f\|_{K_{\chi,[d]}}.
\end{equation*}

Thus we have shown the following lemma.
\begin{lemma}\label{lem_sigma_bound}
We have for all $f\in H(K_{\chi, [d]})$ that
\begin{equation*}
\sigma_{r, (\bsl_u, \bszero), d}(f) \le  \gamma_{v(u)}^{1/2} 2^{|v(u)| \max(\chi, r)}  \prod_{j\in u} b^{-\min(\chi, r) \mu(k_j)} \|f\|_{K_{\chi,[d]}}.
\end{equation*}
\end{lemma}

\section*{Acknowledgments}
Josef Dick is supported by a QE2 Fellowship from the Australian Research Council. \\
Michael Gnewuch acknowledges support by the Australian Research Council.

\footnotesize


\end{document}